\documentclass[12pt]{amsart}
\usepackage{amsfonts}
\usepackage{}       
\usepackage{txfonts}
\usepackage{amssymb}
\usepackage{eucal}
\usepackage{graphicx}
\usepackage{amsmath}
\usepackage{amscd}
\usepackage[all]{xy}           
\usepackage{tikz}
\usepackage{amsfonts,latexsym}
\usepackage{xspace}
\usepackage{epsfig}
\usepackage{float}
\usepackage{color}
\usepackage{fancybox}
\usepackage{colordvi}
\usepackage{multicol}
\usepackage{colordvi}
\usepackage[colorlinks,final,backref=page,hyperindex,hypertex]{hyperref}
\usepackage[active]{srcltx} 

\newtheorem{theorem}{Theorem}[section]
\newtheorem{prop}[theorem]{Proposition}
\theoremstyle{definition}
\newtheorem{defn}[theorem]{Definition}
\newtheorem{lemma}[theorem]{Lemma}
\newtheorem{coro}[theorem]{Corollary}
\newtheorem{prop-def}{Proposition-Definition}[section]
\newtheorem{coro-def}{Corollary-Definition}[section]

\newtheorem{exam}{Example}[section]

\newcommand{\nc}{\newcommand}
\nc{\tred}[1]{\textcolor{red}{#1}}
\nc{\tblue}[1]{\textcolor{blue}{#1}}
\nc{\tgreen}[1]{\textcolor{green}{#1}}
\nc{\tpurple}[1]{\textcolor{purple}{#1}}
\nc{\btred}[1]{\textcolor{red}{\bf #1}}
\nc{\btblue}[1]{\textcolor{blue}{\bf #1}}
\nc{\btgreen}[1]{\textcolor{green}{\bf #1}}
\nc{\btpurple}[1]{\textcolor{purple}{\bf #1}}
\nc{\NN}{{\mathbb N}}
\nc{\ncsha}{{\mbox{\cyr X}^{\mathrm NC}}} \nc{\ncshao}{{\mbox{\cyr
X}^{\mathrm NC}_0}}

\newcommand{\efootnote}[1]{}

\renewcommand{\textbf}[1]{}

\newcommand{\delete}[1]{}

\nc{\mlabel}[1]{\label{#1}}  
\nc{\mcite}[1]{\cite{#1}}  
\nc{\mref}[1]{\ref{#1}}  
\nc{\mbibitem}[1]{\bibitem{#1}} 

\delete{
\nc{\mlabel}[1]{\label{#1}  
{\hfill \hspace{1cm}{\small\tt{{\ }\hfill(#1)}}}}
\nc{\mcite}[1]{\cite{#1}{\small{\tt{{\ }(#1)}}}}  
\nc{\mref}[1]{\ref{#1}{{\tt{{\ }(#1)}}}}  
\nc{\mbibitem}[1]{\bibitem[\bf #1]{#1}} 
}

\nc{\opa}{\ast} \nc{\opb}{\odot} \nc{\op}{\bullet} \nc{\pa}{\frakL}
\nc{\arr}{\rightarrow} \nc{\lu}[1]{(#1)} \nc{\mult}{\mrm{mult}}
\nc{\diff}{\mathfrak{Diff}}
\nc{\opc}{\sharp}\nc{\opd}{\natural}
\nc{\ope}{\circ}
\nc{\dpt}{\mathrm{d}}
\nc{\tforall}{\text{ for all }}
\nc{\diam}{alternating\xspace}
\nc{\Diam}{Alternating\xspace}
\nc{\cdiam}{alternating\xspace}
\nc{\Cdiam}{Alternating\xspace}
\nc{\AW}{\mathcal{A}}
\nc{\rba}{Rota-Baxter algebra\xspace}

\nc{\ari}{\mathrm{ar}}

\nc{\lef}{\mathrm{lef}}

\nc{\Sh}{\mathrm{ST}}

\nc{\Cr}{\mathrm{Cr}}

\nc{\st}{{Schr\"oder tree}\xspace}
\nc{\sts}{{Schr\"oder trees}\xspace}

\nc{\vertset}{\Omega} 

\nc{\assop}{\quad \begin{picture}(5,5)(0,0)
\line(-1,1){10}
\put(-2.2,-2.2){$\bullet$}
\line(0,-1){10}\line(1,1){10}
\end{picture} \quad \smallskip}

\nc{\operator}{\begin{picture}(5,5)(0,0)
\line(0,-1){6}
\put(-2.6,-1.8){$\bullet$}
\line(0,1){9}
\end{picture}}

\nc{\idx}{\begin{picture}(6,6)(-3,-3)
\put(0,0){\line(0,1){6}}
\put(0,0){\line(0,-1){6}}
 \end{picture}}

\nc{\pb}{{\mathrm{pb}}}
\nc{\Lf}{{\mathrm{Lf}}}

\nc{\lft}{{left tree}\xspace}
\nc{\lfts}{{left trees}\xspace}

\nc{\fat}{{fundamental averaging tree}\xspace}

\nc{\fats}{{fundamental averaging trees}\xspace}
\nc{\avt}{\mathrm{Avt}}

\nc{\rass}{{\mathit{RAss}}}

\nc{\aass}{{\mathit{AAss}}}

\nc{\vin}{{\mathrm Vin}}    
\nc{\lin}{{\mathrm Lin}}    
\nc{\inv}{\mathrm{I}n}
\nc{\gensp}{V} 
\nc{\genbas}{\mathcal{V}} 
\nc{\bvp}{V_P}     
\nc{\gop}{{\,\omega\,}}     

\nc{\bin}[2]{ (_{\stackrel{\scs{#1}}{\scs{#2}}})}  
\nc{\binc}[2]{ \left (\!\! \begin{array}{c} \scs{#1}\\
    \scs{#2} \end{array}\!\! \right )}  
\nc{\bincc}[2]{  \left ( {\scs{#1} \atop
    \vspace{-1cm}\scs{#2}} \right )}  
\nc{\bs}{\bar{S}} \nc{\cosum}{\sqsubset} \nc{\la}{\longrightarrow}
\nc{\rar}{\rightarrow} \nc{\dar}{\downarrow} \nc{\dprod}{**}
\nc{\dap}[1]{\downarrow \rlap{$\scriptstyle{#1}$}}
\nc{\md}{\mathrm{dth}} \nc{\uap}[1]{\uparrow
\rlap{$\scriptstyle{#1}$}} \nc{\defeq}{\stackrel{\rm def}{=}}
\nc{\disp}[1]{\displaystyle{#1}} \nc{\dotcup}{\
\displaystyle{\bigcup^\bullet}\ } \nc{\gzeta}{\bar{\zeta}}
\nc{\hcm}{\ \hat{,}\ } \nc{\hts}{\hat{\otimes}}
\nc{\barot}{{\otimes}} \nc{\free}[1]{\bar{#1}}
\nc{\uni}[1]{\tilde{#1}} \nc{\hcirc}{\hat{\circ}} \nc{\lleft}{[}
\nc{\lright}{]} \nc{\lc}{\lfloor} \nc{\rc}{\rfloor}
\nc{\curlyl}{\left \{ \begin{array}{c} {} \\ {} \end{array}
    \right .  \!\!\!\!\!\!\!}
\nc{\curlyr}{ \!\!\!\!\!\!\!
    \left . \begin{array}{c} {} \\ {} \end{array}
    \right \} }
\nc{\longmid}{\left | \begin{array}{c} {} \\ {} \end{array}
    \right . \!\!\!\!\!\!\!}
\nc{\onetree}{\bullet} \nc{\ora}[1]{\stackrel{#1}{\rar}}
\nc{\ola}[1]{\stackrel{#1}{\la}}
\nc{\ot}{\otimes} \nc{\mot}{{{\boxtimes\,}}}
\nc{\otm}{\overline{\boxtimes}} \nc{\sprod}{\bullet}
\nc{\scs}[1]{\scriptstyle{#1}} \nc{\mrm}[1]{{\rm #1}}
\nc{\margin}[1]{\marginpar{\rm #1}}   
\nc{\dirlim}{\displaystyle{\lim_{\longrightarrow}}\,}
\nc{\invlim}{\displaystyle{\lim_{\longleftarrow}}\,}
\nc{\mvp}{\vspace{0.3cm}} \nc{\tk}{^{(k)}} \nc{\tp}{^\prime}
\nc{\ttp}{^{\prime\prime}} \nc{\svp}{\vspace{2cm}}
\nc{\vp}{\vspace{8cm}} \nc{\proofbegin}{\noindent{\bf Proof: }}
\nc{\proofend}{$\blacksquare$ \vspace{0.3cm}}
\nc{\modg}[1]{\!<\!\!{#1}\!\!>}
\nc{\intg}[1]{F_C(#1)} \nc{\lmodg}{\!
<\!\!} \nc{\rmodg}{\!\!>\!}
\nc{\cpi}{\widehat{\Pi}}
\nc{\sha}{{\mbox{\cyr X}}}  
\nc{\shap}{{\mbox{\cyrs X}}} 
\nc{\shan}{{\overrightarrow \sha}}
\nc{\shat}{{\sha}}
\nc{\shpr}{\diamond}    
\nc{\shp}{\ast}
\nc{\shplus}{\shpr^+}
\nc{\shprc}{\shpr_c}    
\nc{\msh}{\ast} \nc{\zprod}{m_0} \nc{\oprod}{m_1}
\nc{\vep}{\varepsilon} \nc{\labs}{\mid\!} \nc{\rabs}{\!\mid}
\nc{\sqmon}[1]{\langle #1\rangle}

\nc{\mmbox}[1]{\mbox{\ #1\ }} \nc{\dep}{\mrm{dep}} \nc{\fp}{\mrm{FP}}
\nc{\rchar}{\mrm{char}} \nc{\End}{\mrm{End}} \nc{\Fil}{\mrm{Fil}}
\nc{\Mor}{Mor\xspace} \nc{\gmzvs}{gMZV\xspace}
\nc{\gmzv}{gMZV\xspace} \nc{\mzv}{MZV\xspace}
\nc{\mzvs}{MZVs\xspace} \nc{\Hom}{\mrm{Hom}} \nc{\id}{\mrm{id}}
\nc{\im}{\mrm{im}} \nc{\incl}{\mrm{incl}} \nc{\map}{\mrm{Map}}
\nc{\mchar}{\rm char} \nc{\nz}{\rm NZ} \nc{\supp}{\mathrm Supp}

\nc{\Alg}{\mathbf{Alg}} \nc{\Bax}{\mathbf{Bax}} \nc{\bff}{\mathbf f}
\nc{\bfk}{{\bf k}} \nc{\bfone}{{\bf 1}} \nc{\bfx}{\mathbf x}
\nc{\bfy}{\mathbf y}
\nc{\base}[1]{\bfone^{\otimes ({#1}+1)}} 
\nc{\Cat}{\mathbf{Cat}}

\nc{\detail}{\marginpar{\bf More detail}
    \noindent{\bf Need more detail!}
    \svp}
\nc{\Int}{\mathbf{Int}} \nc{\Mon}{\mathbf{Mon}}
\nc{\rbtm}{{shuffle }} \nc{\rbto}{{Rota-Baxter }}
\nc{\remarks}{\noindent{\bf Remarks: }} \nc{\Rings}{\mathbf{Rings}}
\nc{\Sets}{\mathbf{Sets}} \nc{\wtot}{\widetilde{\odot}}
\nc{\wast}{\widetilde{\ast}} \nc{\bodot}{\bar{\odot}}
\nc{\bast}{\bar{\ast}} \nc{\hodot}[1]{\odot^{#1}}
\nc{\hast}[1]{\ast^{#1}} \nc{\mal}{\mathcal{O}}
\nc{\tet}{\tilde{\ast}} \nc{\teot}{\tilde{\odot}}
\nc{\oex}{\overline{x}} \nc{\oey}{\overline{y}}
\nc{\oez}{\overline{z}} \nc{\oef}{\overline{f}}
\nc{\oea}{\overline{a}} \nc{\oeb}{\overline{b}}
\nc{\weast}[1]{\widetilde{\ast}^{#1}}
\nc{\weodot}[1]{\widetilde{\odot}^{#1}} \nc{\hstar}[1]{\star^{#1}}
\nc{\lae}{\langle} \nc{\rae}{\rangle}
\nc{\lf}{\lfloor}
\nc{\rf}{\rfloor}

\nc{\QQ}{{\mathbb Q}}
\nc{\RR}{{\mathbb R}} \nc{\ZZ}{{\mathbb Z}}

\nc{\cala}{{\mathcal A}} \nc{\calb}{{\mathcal B}}
\nc{\calc}{{\mathcal C}}
\nc{\cald}{{\mathcal D}} \nc{\cale}{{\mathcal E}}
\nc{\calf}{{\mathcal F}} \nc{\calg}{{\mathcal G}}
\nc{\calh}{{\mathcal H}} \nc{\cali}{{\mathcal I}}
\nc{\call}{{\mathcal L}} \nc{\calm}{{\mathcal M}}
\nc{\caln}{{\mathcal N}}\nc{\calo}{{\mathcal O}}
\nc{\calp}{{\mathcal P}} \nc{\calr}{{\mathcal R}}
\nc{\cals}{{\mathcal S}} \nc{\calt}{{\mathcal T}}
\nc{\calu}{{\mathcal U}} \nc{\calw}{{\mathcal W}} \nc{\calk}{{\mathcal K}}
\nc{\calx}{{\mathcal X}} \nc{\CA}{\mathcal{A}}

\nc{\fraka}{{\mathfrak a}} \nc{\frakA}{{\mathfrak A}}
\nc{\frakb}{{\mathfrak b}}
\nc{\frakc}{{\mathfrak c}}
\nc{\frakd}{{\mathfrak d}}
\nc{\frakB}{{\mathfrak B}}
\nc{\frakD}{{\mathfrak D}} \nc{\frakF}{\mathfrak{F}}
\nc{\frakf}{{\mathfrak f}} \nc{\frakg}{{\mathfrak g}}
\nc{\frakH}{{\mathfrak H}} \nc{\frakL}{{\mathfrak L}}
\nc{\frakM}{{\mathfrak M}} \nc{\bfrakM}{\overline{\frakM}}
\nc{\frakm}{{\mathfrak m}} \nc{\frakP}{{\mathfrak P}}
\nc{\frakN}{{\mathfrak N}} \nc{\frakp}{{\mathfrak p}}
\nc{\frakS}{{\mathfrak S}} \nc{\frakT}{\mathfrak{T}}
\nc{\frakX}{{\mathfrak X}} \nc{\frakx}{\mathfrak{x}}

\nc{\BS}{\mathbb{S
}}

\font\cyr=wncyr10 \font\cyrs=wncyr7
\nc{\li}[1]{\textcolor{red}{#1}}
\nc{\lir}[1]{\textcolor{red}{Li:#1}}

\nc{\sz}[1]{\textcolor{blue}{SZ: #1}}

\nc{\ID}{\mathfrak{I}} \nc{\lbar}[1]{\overline{#1}}
\nc{\bre}{{\rm b}} \nc{\sd}{\cals} \nc{\rb}{\rm RB}
\nc{\A}{\rm angularly decorated\xspace} \nc{\LL}{\rm L}
\nc{\w}{\rm wid} \nc{\arro}[1]{#1}
\nc{\ver}{\rm ver}
\nc{\FN}{F_{\mathrm R}}
\nc{\FR}{F_{\mathrm L}}
\nc{\FNA}{\FN(A)} \nc{\NA}{N_{A}}
\nc{\dr}{\diamond_\lambda}
\nc{\shar}{{\mbox{\cyrs X}}_r} 
\nc{\shrp}{\sha^+} 
\nc{\dt}{\Delta_\lambda}
\nc{\da}{\Delta_A}
\nc{\vt}{\vep_\lambda }
\nc{\ml}{\mu_\lambda}
\nc{\pl}{P_\lambda}
\nc{\bul}{\bullet}
\nc{\fraku}{{\mathfrak U}}
\nc{\rt}{\rm $\lambda$-TD }
\nc{\rto}{\rm P}
\nc{\lam}{\Lambda}

\begin{document}

\title[$\lambda$-TD algebras, generalized shuffle products and left counital Hopf algebras]{$\lambda$-TD algebras, generalized shuffle products and left counital Hopf algebras}
%
\author{Hengyi Luo}
\address{Department of Mathematics, Jiangxi Normal University, Nanchang, Jiangxi 330022, China}
\email{2011713958@qq.com}

\author{Shanghua Zheng}
\address{Department of Mathematics, Jiangxi Normal University, Nanchang, Jiangxi 330022, China}
\email{zhengsh@jxnu.edu.cn}

\date{\today}
\begin{abstract}
The theory of operated algebras has played a pivotal role in mathematics and physics. In this paper, we introduce a $\lambda$-TD algebra that appropriately includes both the Rota-Baxter algebra and the TD-algebra. The explicit construction of free commutative $\lambda$-TD algebra on a commutative algebra is obtained by generalized shuffle products, called  $\lambda$-TD shuffle products. We then show that the free commutative $\lambda$-TD algebra possesses a left counital bialgera structure by means of a suitable 1-cocycle condition. Furthermore, the classical result that every connected filtered bialgebra is a Hopf algebra, is extended to the context of left counital  bialgebras. Given this result, we finally prove that the left counital bialgebra on the free commutative $\lambda$-TD algebra is connected and filtered, and thus is a left counital Hopf algebra.
\end{abstract}

\keywords{Rota-Baxter  algebra, TD algebra, filtered bialgebra,  generalized shuffle product,   left counital Hopf algebra\\
\qquad 2020 Mathematics Subject Classification. 17B38, 16W70, 16S10, 16T05}

\maketitle

\tableofcontents

\setcounter{section}{0}

\allowdisplaybreaks

\section{Introduction}
An algebra with (one or more) linear operators, first appeared in ~\cite{Ku} in the 1960s, is vital in the recent developments in a wide range of areas. The notion of an operated algebra (that is, an associative algebra with only one linear operator) was proposed by Guo for constructing the free Rota-Baxter algebra~\cite{Gop}. A {\bf Rota-Baxter algebra of weight $\lambda$} (also called a {\bf $\lambda$-Rota-Baxter algebra}) is an associative algebra $R$ equipped with a linear operator $P:R\to R$ satisfying the {\bf Rota-Baxter equation}
\begin{equation}
P(x)P(y)=P(xP(y)+P(x)y+\lambda xy),\,\tforall\, x,y\in R.
\mlabel{eq:RB}
\end{equation} 
Then $P$ is called a {\bf Rota-Baxter operator of weight $\lambda$} or a {\bf $\lambda$-Rota-Baxter operator}, where $\lambda$ is a constant. Derived from the work of Baxter on probability study~\cite{Ba}, the Rota-Baxter operator is intimately connected with the classical Yang-Baxter equation~\cite{Bai}, number theory~\cite{G-Z}, combinatorics~\cite{Gop,Ro,Ro2, YGT}, and most conspicuously,  renormalization in quantum field theory based on the Hopf algebra framework of Connes and Kreimer~\cite{EGK,EGM}.

On the other hand, the direct connection between Rota-Baxter algebras and dendriform algebras was first provided by Aguiar~\cite{Ag}, who proved that every Rota-Baxter algebra of weight $0$ naturally gives a  dendriform algebra. Likewise,  Ebrahimi-Fard and Guo showed that every Rota-Baxter algebra of non-zero weight $\lambda$ carries a tridendriform algebra structure~\cite{EG}. In order to offer another way to produce the tridendriform algebra, the TD operator, which can be formally viewed as an  analog of the Rota-Baxter operator, was invented by Leroux~\cite{Le}. A {\bf TD operator} $P: R\to R$ is a linear operator satisfying the {\bf TD equation}
\begin{equation}
P(x)P(y)=P(xP(y)+P(x)y-xP(1)y),\,\tforall\, x,y\in R.
\mlabel{eq:TD}
\end{equation}
Note that when $\lambda$ takes $-P(1)$ in Eq.~(\mref{eq:RB}), the Rota-Baxter opeator becomes a TD operator. From this viewpoint, the TD operator can be considered as a special  class of Rota-Baxter operators. Indeed, the TD operator  belongs to the category of Rota-Baxter type operator~\cite{ZGGS}, which was proposed for solving the Rota's problem of classifying all linear operators on an associative algebra.

In recent years, classical operators, such as differential operators,  Rota-Baxter operators and TD operators, are generalized in diverse ways for developing the various  algebraic structures and phenomena~\cite{G-K3,ZGGS}.  For instance,  The $\lambda$-different operator was introduced by Guo and Keigher in ~\cite{G-K3} for uniformly studying the algebraic structure with both a differential operator and a difference operator, and for the same reason the $\lambda$-differential Rota-Baxter operator was discovered.  Subsequently, in ~\cite{Agg}, the concept of {\bf Rota-Baxter Nijenhuis TD operaotrs} or {\bf RBNTD operators} was represented as a combination of Rota-Baxter operator, Nijenhuis operator and TD operator, giving rise to a RBNTD-dendriform algebra and a five-part splitting of associativity. The Rota-Baxter Nijenhuis TD operator is  defined by the {\bf Rota-Baxter Nijenhuis TD equation}
\begin{equation}
P(x)P(y)=P(xP(y)+P(x)y+\lambda xy-P(xy)-xP(1)y),\quad\tforall x,y\in R.
\mlabel{eq:rbntd}
\end{equation}
Lately, Zhou and Guo~\cite{ZhG17} introduced the concept of a {\bf Rota-Baxter TD operator}, given by the {\bf Rota-Baxter TD equation}
\begin{equation}
P(x)P(y)=P(xP(y)+P(x)y+\lambda xy-xP(1)y-xyP(1)),\quad\tforall x,y\in R.
\mlabel{eq:rbtd}
\end{equation}
As a consequence, a Rota-Baxter TD operator also gives  a five-part splitting of the associativity and induces a quinquedendriform algebra structure.
One can see at once that every TD operator contains the Rota-Baxter operator by taking $P(1)=0$. But if we require $\lambda=0$, the Rota-Baxter TD operator is not a TD operator in general. In this work, we will mainly develop a more appropriate fusion of a Rota-Baxter operator and a TD operator, called a {\bf $\lambda$-TD operator} or a {\bf TD operator of weight $\lambda$} (See Definition~\mref{defn:LTD}).

It is well-known that one of the most meaningful examples of Hopf algebras for applications in mathematical physics is  the Connes-Kreimer Hopf algebra of rooted forests~\cite{CK2,CK}, whose coproduct satisfies the 1-cocycle property. Lately, Hopf algebraic structures on the free non-commutative  Rota-Baxter algebra of decorated rooted forests has been achieved by the same way~\cite{ZGG16}. Furthermore, it is worth mentioning that the explicit constructions of free non-commutative TD algebras and free non-commutative Rota-Baxter TD algebras were also accomplished  by using the rooted trees~\cite{ZhG17,Zhou12}. Based on the construction of  shuffle product Hopf algebras, a Hopf algebra structure was established on the free  commutative (modified) Rota-Baxter algebra by means of  various generalized shuffle products~\cite{AGKO,EG1,GK1,ZGG19}. Motivated by this, in~\cite{GLZ,ZG}, the Hopf algebraic structure on free commutative  and non-commutative Nijenhuis algebras was considered spontaneously. However, it turns out that this method can not produce a genuine Hopf algebra again, only one with a left-sided counit and right-sided antipode. Such Hopf-type algebra is called a {\bf left counital Hopf algebra} (See Definition~\mref{defn:lcHopf}) to distinguish between this Hopf algebra and the usual Hopf algebra. Intriguingly, algebra structures associated with it have already occurred in the study of quantum group~\cite{GNT,RT} and combinatorics~\cite{FLS,FS}. See~\cite{Li,Wa} for other variants of Hopf algebra under more weaker conditions.

Thanks to Rota-Baxter operators, TD operators and Nijenhuis operators~\cite{LG,LSZB} sharing analogous properties and similar applications, it  is reasonable to speculate that the free $\lambda$-TD algebras should possess a weakened form of Hopf algebra structure. In this paper, we primarily aim to  equip the  free commutative $\lambda$-TD algebra with a left counital Hopf algebra structure.

The paper is organized as follows. First of all, in Section~2, we give the concept of $\lambda$-TD algebras and then provide some general properties of $\lambda$-TD algebras in parallel to that of $\lambda$-Rota-Baxter algebras~\cite{Gub}. Then we combine the quasi-shuffle product and the left-shift shuffle product~\cite{EL} together, thus yielding a $\lambda$-TD shuffle product. This allows us to construct the free commutative $\lambda$-TD algebra on a commutative algebra. In Section~3,  we recall the concepts of left counital coalgebra and  left counital bialgebra. Then  applying a proper 1-cocycle property gives a coproduct on the free commutative $\lambda$-TD algebra. Afterwards, a left counit is also defined on it. Thus the free commutative $\lambda$-TD algebra possesses a left counital bialgebra. Finally in Section~4,  we first prove that every connected filtered left counital operated bialgebra is a Hopf algebra. We then show that the aforementioned left counital bialgebra on the free commutative $\lambda$-TD algebra satisfies the connectedness and has an increasing filtration, and thus leads to a left counital Hopf algebra.

\noindent
{\bf Convention. } In this paper, all algebras are taken to be unitary commutative over a unitary commutative ring $\bfk$ unless otherwise specified. Also linear maps and tensor products are taken over $\bfk$.

\section{Free commutative $\lambda$-TD algebras on a commutative algebra}
\label{sec:FreeCNij}

In this section, we  first present a more opportune combination of Rota-Baxter algebras and TD algebras, which can be regarded as one of Rota-Baxter type algebras.  Then  the general properties of $\lambda$-TD lagebas are developed.  The construction of free commutative $\lambda$-TD algebras will be given by the $\lambda$-TD shuffle product, as a generalization of shuffle product~\mcite{EL}.
\subsection{General properties of $\lambda$-TD algebras}

\begin{defn}Let $\lambda\in \bfk$.
A {\bf $\lambda$-TD algebra} is an  algebra $R$ equipped with a linear operator $P$, called a {\bf
$\lambda$-TD operator}, satisfying the {\bf $\lambda$-TD equation}:
\begin{equation}
P(x)P(y)=P(xP(y)+P(x)y+\lambda xy-xP(1)y),\quad\tforall\, x,y\in R.
\mlabel{eq:LTD}
\end{equation}
\mlabel{defn:LTD}
\end{defn}
Formally, every $\lambda$-TD-operator can be obtained from a Rota-Baxter operator by adding  the last term $-P(xP(1)y)$ on the right hand side of Eq.~(\mref{eq:TD})  to the right hand side of Eq.~(\mref{eq:RB}). Note that every TD operator is a $0$-TD operator. From this viewpoint, a $\lambda$-TD operator can be viewed as a natural generalization of TD-operator. Furthermore, we have
\begin{prop} Let $P$ be a linear operator on an algebra $R$. Let $\lambda\in \bfk$ be given.
\begin{enumerate}
\item
If $P(1)=0$, then $P$ is  a $\lambda$-TD operator if and only if $P$ is a Rota-Baxter operator of weight $\lambda$.
\mlabel{it:ltd1}
\item
If  $P(1)=\lambda$, then $P$ is  a $\lambda$-TD operator if and only if $P$ is a Rota-Baxter operator of weight $0$.
\mlabel{it:ltd3}
\item
If  $P(1)=2\lambda$, then $P$ is  a $\lambda$-TD operator if and only if $P$ is a Rota-Baxter operator of weight $-\lambda$.
\mlabel{it:ltd2}
\end{enumerate}
\end{prop}
\begin{proof}
Items~(\mref{it:ltd1}),~(\mref{it:ltd3}) and ~(\mref{it:ltd2}) follow from Eqs.~(\mref{eq:RB}) and~(\mref{eq:LTD}).
\end{proof}
By Item~(\mref{it:ltd2}) and ~\cite[Proposition 2.2]{ZhG17}, if a $\lambda$-TD operator  $P$  satisfies  $P(1)=2\lambda$, then $P$ is a $\lambda$-RBTD operator.
By \cite[Proposition~1.1.12]{Gub}, a $\lambda$-Rota-Baxter operator $P$ leads to another $\lambda$-Rota-Baxter operator $-\lambda\id-P$. However, it is not necessarily true that if $P$ is a $\lambda$-TD operator, so is $-\lambda\id-P$.
From the $\lambda$-TD equation, we obtain
\begin{prop}
Let $P$ is a linear operator on a $\bfk$-algebra $R$. Then $P$ is a $\lambda$-TD operator if and only if $-P$ is a $-\lambda$-TD operator.
\end{prop}

\begin{defn}
Let $P$ be a linear operator on $R$.  Then $P$ is called a {\bf $\lambda$-modified TD operator} if $P$ satisfies the {\bf $\lambda$-modified TD equation}
\begin{equation}
P(x)P(y)=P(xP(y)+P(x)y+\lambda xy)-xP(1)y.
\mlabel{eq:mtd}
\end{equation}
In this case, we call $(R,P)$  a {\bf $\lambda$-modified TD algebra}.
\end{defn}
Every Rota-Baxter algebra contains naturally a double structure, which is intimately related  to the splitting of associativity in algebras such as Loday type algebras, including dendriform algebras and tridendriform algebras~\mcite{Gub}. To explore this structure on  a $\lambda$-TD algebra $(R,P)$, we define another operation $\ast_\lambda$ on $R$, given by
\begin{equation}
x\ast_\lambda y =xP(y)+P(x)y+\lambda xy-xP(1)y.
\mlabel{eq:ast}
\end{equation}
We can prove that $(R,\ast_\lambda,P)$ is not a  $\lambda$-TD algebra by a direct calculation, that is, a $\lambda$-TD algebra does not have the double structure in general. But if the $\lambda$-TD operator $P$ satisfies $P^2=P$, then $(R,\ast_\lambda,P)$ is a $\lambda$-TD algebra. Furthermore, we obtain the following observation.

\begin{prop}
Let $(R,P)$ be a $\lambda$-TD algebra. Then
\begin{enumerate}
\item
The pair $(R,\ast_\lambda)$ is a nonunitary associative algebra;
\mlabel{it:nonass}
\item
The triple $(R,\ast_\lambda,P)$ is a $\lambda$-modified TD algebra.
\mlabel{it:mltd}
\end{enumerate}
\end{prop}
\begin{proof}
(\mref{it:nonass}) follows from ~\cite[Proposition 2.37]{ZGGS}.
\smallskip 

\noindent
(\mref{it:mltd}) By Eqs.~(\mref{eq:LTD}) and~(\mref{eq:ast}), we obtain
\begin{eqnarray*}
P(x)\ast_\lambda P(y)&=&P(x)P^2(y)+P^2(x)P(y)+\lambda P(x)P(y)-P(x)P(1)P(y)\\
&=&P(x)P^2(y)+P^2(x)P(y)+\lambda P(x)P(y)-(P^2(x)+\lambda P(x))P(y)\\
&=&P(x)P^2(y)\\
&=&P(P(x)P(y)).
\end{eqnarray*}
On the other hand, by Item~(\mref{it:nonass}) and Eqs.~(\mref{eq:LTD}) and~(\mref{eq:ast}), we get
\begin{eqnarray*}
&&P\Big(x\ast_\lambda P(y)+P(x)\ast_\lambda y+\lambda x\ast_\lambda y\Big)-x\ast_\lambda P(1)\ast_\lambda y\\
&=&P\Big(P(x)P(y)+P(x)P(y)\Big)+\lambda P(x)P(y)-(P^2(x)P(y)+\lambda P(x)P(y))\\
&=&P(P(x)P(y)+P(x)P(y))-P(P(x)P(y))\\
&=&P(P(x)P(y)).
\end{eqnarray*}
This gives
$$P(x)\ast_\lambda P(y)=P\Big(x\ast_\lambda P(y)+P(x)\ast_\lambda y+\lambda x\ast_\lambda y\Big)-x\ast_\lambda P(1)\ast_\lambda y,$$
proving Item~(\mref{it:mltd}).
\end{proof}

\subsection{The construction of free commutative $\lambda$-TD algebras}

The explicit construction of  free  commutative  TD algebras on a commutative algebra $A$ was carried out in ~\mcite{EL} by using  generalized shuffle products. This section will investigate the construction of the free commutative $\lambda$-TD algebra on $A$ by another generalized shuffle product. We first give the notion of the free commutative $\lambda$-TD algebra on a commutative algebra.
\begin{defn}
{\rm
Let $A$ be a commutative algebra. A free commutative \rt
algebra on $A$ is a commutative \rt algebra $\FR(A)$ with a
\rt operator $P_L$ and an algebra homomorphism $j_A:
A\to \FR(A)$ such that, for any commutative \rt algebra $(R,P)$ and
any  algebra homomorphism $f:A\to R$, there is a unique
 \rt algebra homomorphism $\free{f}: \FR(A)\to R$
such that $f=\free{f}\circ j_A$, that is, the following diagram
$$ \xymatrix{ A \ar[rr]^{j_A}\ar[drr]^{f} && \FR(A) \ar[d]_{\free{f}} \\
&& R}
$$
commutes.}
\mlabel{defn:crbtd}
\end{defn}
For a given unital commutative algebra $A$ with unit $1_A$,
the free commutative Rota-Baxter algebra on $A$  is  given by the quasi-shuffle or mixable shuffle product in ~\mcite{Gub}, and the free commutative TD algebra on $A$ is given by the left-shift shuffle product in \mcite{EL}.

Let
\begin{equation}
\shrp(A):=\bigoplus_{n\geq 0}A^{\ot n}. \notag
\end{equation}
Here $A^{\ot n}$ is the $n$-th tensor power of $A$   with the convention that $A^{\ot 0}=\bfk $.
We next generalize the quasi-shuffle product and left-shift shuffle product by combining them together.
For $\fraka=a_1\ot \cdots \ot a_m\in A^{\ot m}$ and $\frakb=b_1\ot \cdots \ot b_n\in A^{\ot n}$ with $m,n\geq0$, denote $\fraka'=a_2\ot \cdots \ot a_m$ if $m\geq 1$ and $\frakb'=b_2\ot \cdots \ot b_n$ if $n\geq 1$, so that $\fraka =a_1\ot \fraka'$ and $\frakb=b_1\ot \frakb'$.
Define a binary operation $\shar$ on $\shrp(A)$  as follows.  If $m=0$ or $n=0$, that is, $\fraka=c\in \bfk$ or $\frakb=c\in \bfk$, we define $\fraka\shar \frakb$ to be the scalar product: $\fraka\shar \frakb =c\frakb$ or $\fraka \shar\frakb =c\fraka$. If $m\geq 1$ and $n\geq 1$,  we define
\begin{equation}
\fraka\shar \frakb =
a_1\ot ( \fraka' \shar \frakb) + b_1\ot (\fraka \shar \frakb')
+\lambda a_1b_1\ot(\fraka'\shar \frakb')-  a_1b_1\ot((\fraka' \shar 1_A)\shar \frakb').
\mlabel{eq:dfndr}
\end{equation}
Then we extend the product of two pure tensors to a binary operation on $\shrp(A)$ by bilinearity, called the {\bf $\lambda$-TD shuffle product}.
\begin{exam}
Let $\fraka=a_1$ and $\frakb=b_1\ot b_2$. Then
\begin{eqnarray*}
\fraka\shar \frakb&=&a_1\ot b_1\ot b_2 +b_1\ot (a_1\shar b_2)+\lambda a_1b_1\ot b_2-a_1b_1\ot (1_A\shar b_2)\\
&=&a_1\ot b_1\ot b_2+b_1\ot( a_1\ot b_2+b_2\ot a_1+\lambda a_1b_2-a_1b_2\ot 1_A)\\
&&+\lambda a_1b_1\ot b_2-a_1b_1\ot(1_A\ot b_2+b_2\ot 1_A+\lambda b_2- b_2\ot 1_A) \\
&=&a_1\ot b_1\ot b_2+b_1\ot a_1\ot b_2+b_1\ot b_2\ot a_1+\lambda b_1\ot a_1b_2\\
&&-b_1\ot a_1b_2\ot 1_A-a_1b_1\ot 1_A\ot b_2.
\end{eqnarray*}
\end{exam}

We next give some properties of $\lambda$-TD shuffle product $\shar$ for proving that it  satisfies the commutativity and associativity.
\begin{lemma} Let $\fraka=a_1\ot\cdots\ot a_n\in A^{\ot n}$. Then
\begin{equation}
1_A\shar \fraka=\fraka\shar 1_A=\left\{
\begin{array}{lll}
c1_A, & \text{if}\, \fraka=c\in A^{\ot 0};\\
1_A\ot \fraka +\lambda \fraka. &\text{if}\, \fraka\in A^{\ot n}\,\, \text{for}\,\, n\geq 1.
\end{array}\right.
\mlabel{eq:1a}
\end{equation}
\mlabel{lem:1a}
\end{lemma}

\begin{proof}
Let $\fraka\in A^{\ot n}$. For $n=0$, let $\fraka=c\in \bfk$. Then $1_A\shar \fraka=c 1_A=\fraka \shar 1_A$ by the definition of $\shar$. For $n\geq 1$, we let $\fraka=a_1\ot \fraka'$, where $\fraka'=a_2\ot \cdots\ot a_n$.  Then by Eq.~(\mref{eq:dfndr})
\begin{eqnarray*}
1_A\shar \fraka&=&1_A\shar (a_1\ot \fraka')\\
&=&1_A\ot a_1\ot\fraka'+a_1\ot (1_A \shar \fraka')+\lambda a_1\ot \fraka'-a_1\ot (1_A\shar \fraka')\\
&=&1_A\ot\fraka+\lambda \fraka.
\end{eqnarray*}
On the other hand,
\begin{eqnarray*}
\fraka\shar 1_A&=&(a_1\ot \fraka')\shar 1_A\\
&=& a_1\ot(\fraka'\shar 1_A)+1_A \ot a_1\ot \fraka'+\lambda a_1\ot \fraka'-a_1\ot (\fraka'\shar 1_A)\\
&=&1_A\ot\fraka+\lambda \fraka.
\end{eqnarray*}
Thus Eq.~(\mref{eq:1a}) follows.
\end{proof}

\begin{lemma}Let $\fraka=a_1\ot a_2\ot \cdots\ot a_m\in A^{\ot m}$ and $\frakb=b_1\ot b_2\ot \cdots\ot b_n\in A^{\ot n}$ for $m,n\geq 1$. Then
\begin{equation}
(1_A\ot \fraka)\shar \frakb=\fraka\shar (1_A\ot \frakb).
\mlabel{eq:ash}
\end{equation}
\end{lemma}
\begin{proof}
We prove the claim by induction on $m+n\geq 2$. For $m+n=2$, we have $m=n=1$, and so
\begin{eqnarray*}
(1_A\ot a_1)\shar b_1 &=& 1_A\ot (a_1\shar b_1)+b_1\ot 1_A\ot a_1+\lambda b_1\ot a_1-b_1\ot (a_1\shar 1_A)\quad(\text{by Eq.~(\mref{eq:dfndr})})\\
&=& 1_A\ot (a_1\shar b_1)+b_1\ot 1_A\ot a_1+\lambda b_1\ot a_1-b_1\ot (1_A\ot a_1+\lambda a_1)\quad(\text{by Eq.~(\mref{eq:1a})})\\
&=&1_A\ot (a_1\shar b_1).
\end{eqnarray*}
Likewise,
\begin{eqnarray*}
a_1\shar (1_A\ot b_1) &=& a_1\ot 1_A\ot b_1+1_A\ot (a_1\shar b_1)+\lambda a_1\ot b_1-a_1\ot (1_A\shar b_1)\quad(\text{by Eq.~(\mref{eq:dfndr})})\\
&=&  a_1\ot 1_A\ot b_1+1_A\ot (a_1\shar b_1)+\lambda a_1\ot b_1-a_1\ot (1_A\ot b_1+\lambda b_1)\quad(\text{by Eq.~(\mref{eq:1a})})\\
&=&1_A\ot (a_1\shar b_1).
\end{eqnarray*}
Thus Eq.~(\mref{eq:ash}) follows.
Assume that the claim holds for $m+n\leq k$ with $k\geq 2$. Consider $m+n=k+1$. Let $\fraka=a_1\ot \fraka'$ and $\frakb=b_1\ot\frakb'$. By Eqs.~(\mref{eq:dfndr}) and ~(\mref{eq:1a}), we get
\begin{eqnarray*}
(1_A\ot \fraka)\shar \frakb&=& 1_A\ot (\fraka\shar\frakb)+b_1\ot ((1_A\ot \fraka)\shar \frakb')+\lambda b_1\ot (\fraka\shar\frakb')- b_1\ot ((\fraka\shar 1_A)\shar \frakb')\\
&=& 1_A\ot (\fraka\shar\frakb)+b_1\ot ((1_A\ot \fraka)\shar \frakb')+\lambda b_1\ot (\fraka\shar\frakb')- b_1\ot (( 1_A\ot\fraka+\lambda\fraka)\shar \frakb')\\
&=& 1_A\ot (\fraka\shar\frakb).
\end{eqnarray*}
On the other hand,
\begin{eqnarray*}
\fraka\shar (1_A\ot \frakb)&=& a_1\ot (\fraka'\shar(1_A\ot\frakb))+1_A\ot (\fraka\shar \frakb)+\lambda a_1\ot (\fraka'\shar\frakb)- a_1\ot ((\fraka'\shar 1_A)\shar \frakb)\\
&=&a_1\ot (\fraka'\shar(1_A\ot\frakb))+1_A\ot (\fraka\shar \frakb)+\lambda a_1\ot (\fraka'\shar\frakb)- a_1\ot ((1_A\ot \fraka'+\lambda \fraka' )\shar\frakb)\\
&=&a_1\ot (\fraka'\shar(1_A\ot\frakb))+ 1_A\ot (\fraka\shar\frakb)-a_1\ot ((1_A\ot \fraka')\shar\frakb)\\
&=&1_A\ot (\fraka\shar\frakb).\quad(\text{by the induction hypothesis})
\end{eqnarray*}
Then Eq.~(\mref{eq:ash}) holds.  Induction on $m+n$  completes the proof of the claim.
\end{proof}
From the proof of the above lemma, we also obtain
\begin{equation}
(1_A\ot\fraka)\shar \frakb=1_A\ot(\fraka \shar\frakb),\quad\text{for all}\, \fraka\in A^{\ot m}, \frakb\in A^{\ot n}.
\mlabel{eq:oneoab}
\end{equation}
\begin{lemma}The \rt shuffle product $\shar$ on $\shrp(A)$ is commutative.
\mlabel{lem:comm}
\end{lemma}
\begin{proof}
It suffices to prove
\begin{equation}
\fraka\,\shar \,\frakb= \frakb\, \shar \,\fraka,
\mlabel{eq:shcom}
\end{equation}
for all pure tensors $\fraka:=a_1\ot a_2\ot\cdots\ot a_m\in A^{\ot m}$ and $\frakb:=b_1\ot b_2\ot\cdots\ot b_n\in A^{\ot n}$ with $m,n\geq 0$.
Use induction on $m+n\geq 0$. If $m=0$, or $n=0$, then $\fraka=c\in \bfk$, or $\frakb=c\in\bfk$, and so $\fraka \shar \frakb=\frakb\shar\fraka $ by the definition of $\shar$.  If $m\geq 1$ and $n\geq 1$, we let $\fraka=a_1\ot\fraka'$ with $\fraka'\in A^{\ot {(m-1)}}$ and $\frakb=b_1\ot\frakb'$ with $\frakb'\in A^{\ot{(n-1)}}$.  Assume that Eq.~(\mref{eq:shcom}) holds for $ m+n\leq k$.  Consider $m+n=k+1$. Then by Eq.~(\mref{eq:dfndr}), we get
\begin{equation*}
  \fraka\, \shar \frakb =
a_1\ot ( \fraka' \shar \frakb) + b_1\ot (\fraka \shar \frakb')
+\lambda a_1b_1\ot(\fraka'\shar \frakb')-  a_1b_1\ot((\fraka' \shar 1_A)\shar \frakb').
\end{equation*}
By Eq.~(\mref{eq:1a}), we obtain $\fraka' \shar 1_A=1_A\ot\fraka'+\lambda\fraka'$, and then using the induction hypothesis, we have
\begin{equation*}
\fraka\, \shar \frakb =
a_1\ot ( \frakb\shar \fraka' ) + b_1\ot (\frakb' \shar \fraka)-  a_1b_1\ot(\frakb'\shar(1_A\ot\fraka')).
\end{equation*}
By Eq.~(\mref{eq:dfndr}) again, we get
\begin{equation*}
  \frakb\,\shar \fraka =
b_1\ot ( \frakb' \shar \fraka) + a_1\ot (\frakb \shar \fraka')
+\lambda b_1a_1\ot(\frakb'\shar \fraka')-  b_1a_1\ot((\frakb' \shar 1_A)\shar \fraka').
\end{equation*}
Applying Eq.~(\mref{eq:1a}) gives $\frakb' \shar 1_A=1_A\ot\frakb'+\lambda\frakb'$, and  so
\begin{equation*}
\frakb\,\shar \fraka =
b_1\ot ( \frakb' \shar \fraka) + a_1\ot (\frakb \shar \fraka')-  b_1a_1\ot((1_A\ot\frakb')\shar \fraka').
\end{equation*}
Then Eq.~(\mref{eq:shcom}) follows from the commutativity of $A$ and Eq.~(\mref{eq:ash}). This completes the induction and the proof of the lemma.
\end{proof}
\begin{lemma}The \rt shuffle product $\shar$ on $\shrp(A)$ is associative.
\mlabel{lem:asso}
\end{lemma}
\begin{proof}
To show that the associativity of $\shar$, we need only prove
\begin{equation}
(\fraka \shar \frakb)\shar \frakc=\fraka\shar(\frakb\shar\frakc),
\mlabel{eq:asso}
\end{equation}
for all pure tensors $\fraka\in A^{\ot m},\frakb\in A^{\ot n},\frakc \in A^{\ot \ell}$ with $m,n,\ell\geq 0$. Use induction on $s:=m+n+\ell\geq 0$.
If one of $m,n,\ell$ is $0$, then
Eq.~(\mref{eq:asso}) is true by the definition of $\shar$. This proves Eq.~(\mref{eq:asso}) for $0\leq s\leq 2$. Assume that Eq.~(\mref{eq:asso}) holds for $s\leq k$ with $k\geq 2$, and consider $s=m+n+\ell=k+1$ with $m,n,\ell\geq 1$. Denote $\fraka=a_1\ot \fraka',\frakb=b_1\ot \frakb'$, and $\frakc=c_1\ot \frakc'$ with $\fraka'\in A^{\ot (m-1)},\frakb'\in A^{\ot (n-1)}, \frakc'\in A^{\ot (\ell-1)}$.
Then we have
\begin{eqnarray*}
(\fraka\shar \frakb)\shar \frakc
 &=&\Big(a_1\ot(\fraka'\shar\frakb)\Big)\shar\frakc+\Big(b_1\ot(\fraka\shar\frakb')\Big) \shar\frakc\\
&\;&+\Big(\lambda a_1b_1\ot(\fraka'\shar\frakb')\Big)\shar\frakc-\bigg(a_1b_1\ot\Big(\big(\fraka'\shar1_A\big)\shar\frakb'\Big)\bigg)\shar\frakc \quad(\text{by Eq.~(\mref{eq:dfndr})})\\
&=&a_1\ot\Big(\big(\fraka'\shar\frakb\big)\shar\frakc\Big)+c_1\ot\bigg(\Big(a_1\ot\big(\fraka'\shar\frakb\big)\Big)\shar\frakc'\bigg)
+\lambda a_1c_1\ot\Big(\big(\fraka'\shar\frakb\big)\shar\frakc'\Big)\\
&&-a_1c_1\ot\bigg(\Big(\big(\fraka'\shar\frakb\big)\shar1_A\Big)\shar\frakc'\bigg)
+b_1\ot\Big(\big(\fraka\shar\frakb'\big)\shar\frakc\Big)
+c_1\ot\bigg(\Big(b_1\ot\big(\fraka\shar\frakb'\big)\Big)\shar\frakc'\bigg)\\
&&+\lambda b_1c_1\ot\Big(\big(\fraka\shar\frakb'\big)\shar\frakc'\Big)-b_1c_1\ot\bigg(\Big(\big(\fraka\shar\frakb'\big)\shar1_A\Big)\shar\frakc'\bigg)
+\lambda a_1b_1\ot\Big(\big(\fraka'\shar\frakb'\big)\shar \frakc\Big)\\
&&+c_1\ot\Big(\Big(\lambda a_1b_1\ot (\fraka'\shar \frakb')\Big)\shar \frakc'\Big)
+\lambda^2 a_1b_1c_1\ot \Big(\big(\fraka'\shar\frakb'\big)\shar\frakc'\Big)\\
&&-\lambda a_1b_1c_1\ot\bigg(\Big(\big(\fraka'\shar\frakb'\big)\shar 1_A\Big)\shar\frakc'\bigg)
-a_1b_1\ot\bigg(\Big(\big(\fraka'\shar1_A\big)\shar\frakb'\Big)\shar\frakc\bigg)\\
&&-c_1\ot\bigg(\Big(a_1b_1\ot\Big(\big(\fraka'\shar1_A\big)\shar\frakb'\Big)\Big)\shar\frakc'\bigg)
-\lambda a_1b_1c_1\ot\bigg(\Big(\big(\fraka'\shar1_A\big)\shar\frakb'\Big)\shar\frakc'\bigg)\\
&&+ a_1b_1c_1\ot\bigg(\bigg(\Big(\big(\fraka'\shar1_A\big)\shar\frakb'\Big)\shar 1_A\bigg)\shar\frakc'\bigg).
\end{eqnarray*}
Combining the second, sixth, tenth and fourteenth terms and by Eq.~(\mref{eq:dfndr}), we obtain
\begin{eqnarray*}
(\fraka\shar \frakb)\shar \frakc
&=&a_1\ot\Big(\big(\fraka'\shar\frakb\big)\shar\frakc\Big)+c_1\ot\bigg(\big(\fraka\shar\frakb\big)\shar\frakc'\bigg)
+\lambda a_1c_1\ot\Big(\big(\fraka'\shar\frakb\big)\shar\frakc'\Big)\\
&&-a_1c_1\ot\bigg(\Big(\big(\fraka'\shar\frakb\big)\shar1_A\Big)\shar\frakc'\bigg)
+b_1\ot\Big(\big(\fraka\shar\frakb'\big)\shar\frakc\Big)
+\lambda b_1c_1\ot\Big(\big(\fraka\shar\frakb'\big)\shar\frakc'\Big)\\
&&-b_1c_1\ot\bigg(\Big(\big(\fraka\shar\frakb'\big)\shar1_A\Big)\shar\frakc'\bigg)
+\lambda a_1b_1\ot\Big(\big(\fraka'\shar\frakb'\big)\shar \frakc\Big)
+\lambda^2 a_1b_1c_1\ot \Big(\big(\fraka'\shar\frakb'\big)\shar\frakc'\Big)\\
&&-\lambda a_1b_1c_1\ot\bigg(\Big(\big(\fraka'\shar\frakb'\big)\shar 1_A\Big)\shar\frakc'\bigg)
-a_1b_1\ot\bigg(\Big(\big(\fraka'\shar1_A\big)\shar\frakb'\Big)\shar\frakc\bigg)
\\
&&-\lambda a_1b_1c_1\ot\bigg(\Big(\big(\fraka'\shar1_A\big)\shar\frakb'\Big)\shar\frakc'\bigg)+ a_1b_1c_1\ot\bigg(\bigg(\Big(\big(\fraka'\shar1_A\big)\shar\frakb'\Big)\shar 1_A\bigg)\shar\frakc'\bigg).
\end{eqnarray*}
On the other hand,
\begin{eqnarray*}
\fraka\shar (\frakb \shar \frakc) &=&\fraka\shar\Big(b_1\ot\big(\frakb'\shar\frakc\big)\Big)+\fraka\shar\Big(c_1\ot\big(\frakb\shar\frakc'\big)\Big)\\
&&+\fraka\shar\Big(\lambda b_1c_1\ot\big(\frakb'\shar\frakc'\big)\Big)-\fraka\shar\bigg(b_1c_1\ot\Big(\big(\frakb'\shar1_A\big)\shar\frakc'\Big)\bigg)\quad(\text{by Eq.~(\mref{eq:dfndr})})\\
&=&a_1\ot\bigg(\fraka'\shar\Big(b_1\ot\big(\frakb'\shar\frakc\big)\Big)\bigg)+b_1\ot\Big(\fraka\shar\big(\frakb'\shar\frakc\big)\Big)
+\lambda a_1b_1\ot\Big(\fraka'\shar\big(\frakb'\shar\frakc\big)\Big)\\
&&-a_1b_1\ot\Big(\big(\fraka'\shar1_A\big)\shar\big(\frakb'\shar\frakc  \big)\Big)
+a_1\ot\bigg(\fraka'\shar\Big(c_1\ot\big(\frakb\shar\frakc'\big)\Big)\bigg)
+c_1\ot\Big(\fraka\shar\big(\frakb\shar\frakc'\big)\Big)\\
&&+\lambda a_1c_1\ot\Big(\fraka'\shar\big(\frakb\shar\frakc'\big)\Big)
-a_1c_1\ot\Big(\big(\fraka'\shar1_A\big)\shar\big(\frakb\shar\frakc'  \big)\Big)\\
&&+a_1\ot\bigg(\fraka'\shar\Big(\lambda b_1c_1\ot\big(\frakb'\shar\frakc'\big)\Big)\bigg)
+\lambda b_1c_1\ot\bigg(\fraka\shar\Big(\frakb'\shar\frakc'\Big)\bigg)\\
&&+\lambda^2 a_1b_1c_1\ot\bigg(\fraka'\shar\Big(\frakb'\shar\frakc'\Big)\bigg)-\lambda a_1b_1c_1\ot \bigg(\Big(\fraka'\shar 1_A\Big)\shar\Big(\frakb'\shar\frakc'\Big)\bigg)\\
&&-a_1\ot\bigg(\fraka'\shar \Big(b_1c_1\ot\Big(\big(\frakb'\shar1_A\big)\shar\frakc'\Big)\Big)\bigg)-b_1c_1\ot\bigg(\fraka\shar \Big(\big(\frakb'\shar1_A\big)\shar\frakc'\Big)\bigg)\\
&&-\lambda a_1b_1c_1\ot\bigg(\fraka'\shar \Big(\big(\frakb'\shar1_A\big)\shar\frakc'\Big)\bigg)+a_1b_1c_1\ot\bigg(\Big(\fraka'\shar 1_A\Big)\shar \Big(\big(\frakb'\shar1_A\big)\shar\frakc'\Big)\bigg)
\end{eqnarray*}

Adding the first, fifth, ninth and thirteenth terms and  then using Eq.~(\mref{eq:dfndr}) again, we have
\begin{eqnarray*}
\fraka\shar (\frakb \shar \frakc)
&=&a_1\ot\bigg(\fraka'\shar\big(\frakb\shar\frakc\big)\bigg)+b_1\ot\Big(\fraka\shar\big(\frakb'\shar\frakc\big)\Big)
+\lambda a_1b_1\ot\Big(\fraka'\shar\big(\frakb'\shar\frakc\big)\Big)\\
&&-a_1b_1\ot\Big(\big(\fraka'\shar1_A\big)\shar\big(\frakb'\shar\frakc  \big)\Big)
+c_1\ot\Big(\fraka\shar\big(\frakb\shar\frakc'\big)\Big)
+\lambda a_1c_1\ot\Big(\fraka'\shar\big(\frakb\shar\frakc'\big)\Big)\\
&&-a_1c_1\ot\Big(\big(\fraka'\shar1_A\big)\shar\big(\frakb\shar\frakc'  \big)\Big)
+\lambda b_1c_1\ot\bigg(\fraka\shar\Big(\frakb'\shar\frakc'\Big)\bigg)\\
&&+\lambda^2 a_1b_1c_1\ot\bigg(\fraka'\shar\big(\frakb'\shar\frakc'\big)\bigg)-\lambda a_1b_1c_1\ot \bigg(\Big(\fraka'\shar 1_A\Big)\shar\Big(\frakb'\shar\frakc'\Big)\bigg)\\
&&-b_1c_1\ot\bigg(\fraka\shar \Big(\big(\frakb'\shar1_A\big)\shar\frakc'\Big)\bigg)
-\lambda a_1b_1c_1\ot\bigg(\fraka'\shar \Big(\big(\frakb'\shar1_A\big)\shar\frakc'\Big)\bigg)\\
&&+a_1b_1c_1\ot\bigg(\Big(\fraka'\shar 1_A\Big)\shar \Big(\big(\frakb'\shar1_A\big)\shar\frakc'\Big)\bigg).
\end{eqnarray*}

Applying Eq.~(\mref{eq:1a}) and the induction hypothesis to the seventh term $$-a_1c_1\ot\Big(\big(\fraka'\shar1_A\big)\shar\big(\frakb\shar\frakc'  \big)\Big)$$
gives 
$$-a_1c_1\ot\Big(\big(\big(\fraka'\shar\frakb\big)\shar1_A\big)\shar\frakc'\Big).$$
Then by the induction hypothesis, we obtain
\begin{eqnarray*}
\fraka\shar (\frakb \shar \frakc)
&=&a_1\ot\bigg(\big(\fraka'\shar\frakb\big)\shar\frakc\bigg)+b_1\ot\Big(\big(\fraka\shar\frakb'\big)\shar\frakc\Big)
+\lambda a_1b_1\ot\Big(\big(\fraka'\shar\frakb'\big)\shar\frakc\Big)\\
&&-a_1b_1\ot\bigg(\Big(\big(\fraka'\shar1_A\big)\shar\frakb'\Big)\shar\frakc\bigg)
+c_1\ot\Big(\big(\fraka\shar\frakb\big)\shar\frakc'\Big)
+\lambda a_1c_1\ot\Big(\big(\fraka'\shar\frakb\big)\shar\frakc'\Big)\\
&&-a_1c_1\ot\Big(\Big(\big(\fraka'\shar\frakb\big)\shar 1_A\Big)\shar\frakc'\Big)
+\lambda b_1c_1\ot\bigg(\big(\fraka\shar\frakb'\big)\shar\frakc'\bigg)\\
&&+\lambda^2 a_1b_1c_1\ot\bigg(\big(\fraka'\shar\frakb'\big)\shar\frakc'\bigg)
-\lambda a_1b_1c_1\ot \bigg(\Big(\Big(\fraka'\shar 1_A\Big)\shar\frakb'\Big)\shar\frakc'\bigg)\\
&&-b_1c_1\ot\bigg(\Big(\big(\fraka\shar\frakb'\big)\shar1_A\Big)\shar\frakc'\bigg)
-\lambda a_1b_1c_1\ot\bigg(\Big(\big(\fraka'\shar\frakb'\big)\shar1_A\Big)\shar\frakc'\bigg)\\
&&+a_1b_1c_1\ot\bigg(\bigg(\Big(\big(\fraka'\shar 1_A\big)\shar\frakb'\Big)\shar1_A\bigg)\shar\frakc'\bigg).
\end{eqnarray*}
Then the $i$-th term in the expansion of $(\fraka\shar \frakb)\shar \frakc $ matches with the $\sigma(i)$-th term in the expansion of $\fraka\shar (\frakb \shar \frakc)$, where the permutation $\sigma\in \Sigma_{13}$ is
\begin{equation*}
\left ( \begin{array}{c} i\\\sigma(i)\end{array}\right)
= \left ( \begin{array}{ccccccccccccc} 1&2&3&4&5&6&7&8&9&10&11&12&13\\
    1&5&6&7&2&8&11&3&9&12&4&10&13\end{array} \right ).
\mlabel{eq:s14}
\end{equation*}
This completes the proof.
\end{proof}

\begin{prop}
The triple $(\shrp(A), \shar, 1_\bfk)$ forms a unitary commutative algebra.
\mlabel{prop:shrass}
\end{prop}
\begin{proof}
This follows from Lemma~\mref{lem:comm} and Lemma~\mref{lem:asso}.
\end{proof}

We next construct the free object of the category of \rt algebras on a commutative algebra $A$.  Let
\begin{equation}
\shat(A):=A\ot \shrp(A)=(A\ot \bfk) \oplus A^{\ot 2}\oplus\cdots (\cong \bigoplus_{n\geq 1}A^{\ot n}).
\mlabel{eq:defdr1}
\end{equation}
Here $A^{\ot n}$ is the $n$-th tensor power of $A$.

We first recall the definition of the {\bf right-shift operator $\pl$} on $\shat (A)$. Let $\fraka:=a_0\ot\fraka' \in \shat(A)$ for $\fraka'\in A^{\ot n}$ and all $n\geq 0$. If $n=0$, we let $\fraka'=c\in\bfk(=A^{\ot 0})$. Define
\begin{equation}
\pl:\shat (A)\to \shat (A), \quad \fraka\mapsto 1_A\ot \fraka, \,n\geq 1\quad\text{and}\quad\fraka\mapsto1_A\ot ca_0, \,n=0.
\end{equation}
We next define a multiplication $\dr$ on $\shat (A)$ as follows.  For this purpose, we just need to define the product of two pure tensors and then to  extend by bilinearity. For $\fraka:=a_0\ot\fraka'\in A\ot A^{\ot m}$ and $\frakb:=b_0\ot\frakb' \in A\ot A^{\ot n}$,  we define
\begin{equation}
\fraka\dr \frakb = a_0b_0\ot (\fraka' \shar \frakb'),
\mlabel{eq:dro}
\end{equation}
where $\shar$ is the \rt shuffle product defined in Eq.~(\mref{eq:dfndr}). Then the associativity  and commutativity of $\dr$  follows from  that of the multiplication in  $A$ and $\shar$ in $\shrp(A)$.

Alternatively,  let $ \shat(A)= \bigoplus\limits_{n\geq 1}A^{\ot n}$. For $\fraka=a_1\ot \cdots \ot a_m\in A^{\ot m}$ and $\frakb=b_1\ot\cdots\ot b_n \in A^{\ot n}, m,n\geq 1$, denote $\fraka'=a_2\ot \cdots \ot a_m$ if $m\geq 2$ and $\frakb'=b_2\ot \cdots \ot b_n$ if $n\geq 2$, so that $\fraka =a_1\ot \fraka'$ and $\frakb=b_1\ot \frakb'$.
Then  $\dr$ on $\shat(A)$  can also be defined by the following recursion.
\begin{equation}
\fraka\dr \frakb = \left \{\begin{array}{ll} a_1b_1, & m=n=1, \\
a_1b_1\ot \frakb', & m=1, n\geq 2, \\
a_1b_1\ot \fraka', & m\geq 2, n=1,\\
a_1b_1\ot \Big( \fraka' \dr (1_A\ot \frakb') + (1_A\ot \fraka')\dr \frakb'\\
+\lambda \fraka'\dr\frakb' -  \Big(\fraka' \dr \pl(1_A)\Big)\dr \frakb'\Big), & m, n\geq 2.
\end{array} \right .
\mlabel{eq:dfndr2}
\end{equation}

Let
$$j_A: A\to \shat (A),\,a\mapsto a,$$
be the natural embedding. Then
$$j_A(ab)=ab=a\dr b=j_A(a)\dr j_A(b),\,\tforall\, a,b\in A.$$
So $j_A$ is an algebra homomorphism.
\begin{theorem}
Let $A$ be a commutative algebra. Let $\shat (A),\pl$, $\dr$ and $j_A$ be defined as above.  Then
\begin{enumerate}
\item
The triple $(\shat (A), \dr, \pl)$ is a commutative \rt algebra;
\mlabel{it:crt}
\item
The quadruple~$(\shat (A),\dr,\pl, j_A)$ is the free commutative \rt algebra on $A$.
\mlabel{it:fcrt}
\end{enumerate}
\mlabel{thm:freeCTD}
\end{theorem}
\begin{proof}
(\mref{it:crt})Let $\fraka, \frakb\in\shat(A)$.  Then by Eq.~(\mref{eq:dfndr2}), we have
\begin{eqnarray*}
&&\lefteqn{\pl(\fraka)\dr \pl(\frakb)} \\
&=& (1_A\ot\fraka)\dr(1_A\ot\frakb)\\
&=&1_A\ot\Big(\fraka\dr\big(1_A\ot\frakb\big)\Big)+1_A\ot\Big(\big(1_A\ot\fraka\big)\dr\frakb\Big)+\lambda1_A\ot\big(\fraka\dr\frakb  \big)-1_A\ot\bigg(\Big(\fraka\dr \pl(1_A)\Big)\dr\frakb\bigg)\\
&=&\pl\Big(\fraka\dr \pl(\frakb)\Big)+\pl\Big(\pl(\fraka)\dr\frakb \Big)+\pl\big(\fraka\dr\frakb\big)-\pl\bigg(\Big(\fraka\dr \pl(1_A)\Big)\dr\frakb\bigg).
\end{eqnarray*}
Thus $\pl$ is a $\lambda$-TD operator on $\shat(A)$, and so $(\shat(A),\dr,\pl)$ forms a commutative $\lambda$-TD algebra.
\smallskip

\noindent
(\mref{it:fcrt})
We now show that $(\shat (A),\dr,\pl, j_A)$  is a free commutative $\lambda$-TD algebra, that is, $\shat(A)$ with $j_A$ satisfies the universal property in Definition~\mref{defn:crbtd}. Let $(R,P)$ be a commutative \rt algebra  and let $f:A\to R$ be an
 algebra homomorphism.
For any pure tensor $\fraka=a_{1}\otimes a_{2}\otimes\cdots\otimes a_{m}\in A^{\otimes m}$, we apply the induction on $m$ to define a $\lambda$-TD algebra homomorphism $\bar{f}:\shat(A)\to R $. If $m=1$, we define $\bar{f}(\fraka)=f(\fraka)$. Then $\bar{f}(1_A)=f(1_A)=1_R$, the unit of $R$. Assume that $\bar{f}(\fraka)$ has been defined for $m\leq k$ with $k\geq1$. Consider $\fraka=a_1\ot \fraka'\in A^{\otimes(k+1)}$ for $\fraka'\in A^{\otimes k}$. Note that \begin{equation}\fraka=a_{1}\dr(1_{A}\otimes\fraka')=a_{1}\dr \pl(\fraka').
\mlabel{eq:adr}
\end{equation}
Then define
\begin{equation}
\bar{f}(\fraka) =f(a_{1})P(\bar{f}(\fraka')),
\mlabel{eq:deff}
\end{equation}
where $\bar{f}(\fraka')$ is well-defined by the induction hypothesis. The uniqueness of $\bar{f}$ follows from the definition of $\bar{f}$.

Next we will verify that $\bar{f}$ is a \rt algebra homomorphism. By Eq.~(\mref{eq:deff}), we obtain   $$\bar{f}(\pl(\fraka))=\bar{f}(1_{A}\otimes \fraka)=f(1_A)P(\bar{f}(\fraka))=P(\bar{f}(\fraka)).$$
This gives
\begin{equation}\bar{f}\circ \pl=P\circ \bar{f}.
\mlabel{eq:fpr}
\end{equation}
 So it suffices to  verify that $\bar{f}$ satisfies
\begin{equation}
\bar{f}(\fraka\dr\frakb)=\bar{f}(\fraka)\bar{f}(\frakb),\quad\forall \fraka\in A^{\ot m},\,\frakb\in A^{\ot n}.
\mlabel{eq:comdr}
\end{equation}

We will carry out the verification by induction on $m+n\geq2.$  If $m+n=2$, then $m=n=1$, and so $\fraka, \frakb\in A$. By Eq.~(\mref{eq:deff}), we have
$$\bar{f}(\fraka\dr\frakb)=\bar{f}(\fraka\frakb)=f(\fraka\frakb)=f(\fraka)f(\frakb)=\bar{f}(\fraka)\bar{f}(\frakb).$$
Assume that Eq.~(\mref{eq:comdr}) holds for $m+n\leq k$. Let $\fraka=a_1\ot \fraka'\in A^{\ot m}$ and $\frakb=b_1\ot \frakb'\in A^{\ot n}$ with $m+n=k+1$. Then

\begin{eqnarray*}
\bar{f}(\fraka\dr\frakb)
  &=& \bar{f}\Big(\big(a_1\otimes\fraka'\big)\dr\big(b_1\otimes\frakb'\big)\Big) \\
  &=&\bar{f}\bigg(\Big(a_1\dr \pl\big(\fraka'\big)\Big)\dr\Big(b_1\dr \pl\big(\frakb'\big)\Big)\bigg)\quad(\text{by Eq.~(\mref{eq:adr})}) \\
  &=&\bar{f}\bigg(\big(a_1b_1\big)\dr\Big(\pl\big(\fraka'\big)\dr \pl\big(\frakb'\big)\Big)\bigg)\quad(\text{by the commutativity of $\dr$})\\
  &=&\bar{f}\Bigg(\big(a_1b_1\big)\dr \pl\bigg(\fraka'\dr \pl\big(\frakb'\big)+\pl\big(\fraka'\big)\dr\frakb'\\
&&+\lambda \fraka'\dr\frakb'-\big(\fraka'\dr \pl(1_A)\big)\dr\frakb'\bigg)\Bigg)\quad(\text{by $\pl$ being a $\lambda$-TD operator})\\
  &=&f\big(a_1b_1\big)P\bigg(\bar{f}\Big(\fraka'\dr \pl\big(\frakb'\big)+\pl\big(\fraka'\big)\dr \frakb' \\
  &\;&+\lambda \fraka'\dr\frakb'-(\fraka'\dr \pl(1_A))\dr\frakb'\Big)\bigg)\quad(\text{by Eq.~(\mref{eq:deff})})\\
  &=&f\big(a_1b_1\big)P\bigg(\bar{f}(\fraka')\bar{f}( \pl(\frakb'))+\bar{f}(\pl(\fraka'))\bar{f}( \frakb') \\
  &\;&+\lambda \bar{f}(\fraka')\bar{f}(\frakb')-\Big(\bar{f}(\fraka')\bar{f}(\pl(1_A))\Big)\bar{f}(\frakb')\bigg)\quad(\text{by the induction hypothesis}) \\
  &=&f(a_1)f(b_1)P\bigg(\bar{f}(\fraka')P(\bar{f}(\frakb'))+P(\bar{f}(\fraka'))\bar{f}( \frakb') \\
  &\;&+\lambda \bar{f}(\fraka')\bar{f}(\frakb')-\bar{f}(\fraka')P(\bar{f}(1_A))\bar{f}(\frakb')\bigg)\quad(\text{by Eq.~(\mref{eq:fpr})})\\
  &=&f\big(a_1\big)f\big(b_1\big)P\Big(\bar{f}\big(\fraka'\big)\Big)P\Big(\bar{f}\big(\frakb'\big)\Big)\quad(\text{ by $\bar{f}(1_A)=1_R$ and $P$ being a $\lambda$-TD operator})\\
  &=&\bigg(f\big(a_1\big)P\Big(\bar{f}\big(\fraka'\big)\Big)\bigg)\bigg(f\big(b_1\big)P\Big(\bar{f}\big(\frakb'\big)\Big)\bigg)\quad(\text{by the commutativity of $A$}) \\
  &=&\bar{f}(\fraka)\bar{f}(\frakb).\quad(\text{by Eq.~(\mref{eq:deff})})\\
\end{eqnarray*}
This completes the induction,  and  so the proof of Theorem~\mref{thm:freeCTD}.
\end{proof}

\section{The cocycle bialgebra structure on free commutative \rt algebras}
\mlabel{sec:bialg}

In this section, the free commutative \rt algebra $\shat(A)$ obtained in Theorem~\mref{thm:freeCTD} will be equipped with a  bialgebra structure, under the assumption that the generating algebra $A$ is a  bialgebra. So we let $A:=(A,m_A,\mu_A,\Delta_A,\vep_A)$ be a  bialgebra. To achieve our goal, the first step in this process is to construct a comultiplication on the free commutative \rt algebra $\shat (A):=(\shat (A), \dr, \pl)$ in terms of a suitable $1$-cocycle property $\Delta P=(\id\ot P)\Delta$, which was used to construct left counital Hopf algebras on free Nijenhuis algebras~\mcite{GLZ,ZG} and on bi-decorated planar rooted forests~\mcite{PZGL}.  Afterward,  a left counit on $\shat (A)$ is given.

\subsection{Comultiplication by cocycle condition}

Let us first recall the definition of a left counital cocycle bialgebra.
\begin{defn} \cite{GLZ,ZG,PZGL}
\begin{enumerate}
\item A {\bf left counital coalgebra} is a triple $(C,\Delta,\vep)$, where $C$ is a $\bfk$-module, the comultiplication $\Delta: C\to C\ot C$ is coassociative and the counit $\vep: C \to\bfk$ is left counital, that is, $(\vep\ot \id)\Delta =\beta_\ell$, where $\beta_\ell: C\to\bfk\ot C$, given by $c\mapsto 1\ot c$, is a bijection.
\item
A {\bf left counital operated bialgebra} is a sextuple $(H, m, \mu, \Delta, \vep,P)$,  where the quadruple $(H, m ,\mu, P)$ is an operated algebra and the triple $(H, \Delta, \vep)$ is a left counital coalgebra such that $\Delta:H\to H\ot H$ and $\vep:H\to \bfk$ are algebra homomorphisms;
\item
A left counintal operated bialgebra $(H,m,\mu,\Delta,\vep,P)$ that satisfies the 1-cocycle property $\Delta P=(\id\ot P)\Delta$ is called a {\bf left counital cocycle  bialgebra}.
\end{enumerate}
\mlabel{defn:leftbia}
\end{defn}
In order to distinguish  the multiplication in $\shat(A)$ and in $\shat(A)\ot \shat(A)$,  we  denote by $\bul$ the multiplication in $\shat(A)\ot \shat(A)$.

Let $A:=(A,m_A,\mu_A,\Delta_A,\vep_A)$ be a  bialgebra. Now we begin with the construction of the comultiplicaiton $\dt:\shat (A)\to \shat (A)\ot \shat (A)$. For this, it suffices to define $\dt(\fraka)$ for
$\fraka:=a_1\ot a_2\ot\cdots\ot a_n\in A^{\ot n}$ with $n\geq 1$, and then to extend by linearity. Use induction on $n$, starting with $n=1$, that is, $\fraka=a_1\in A$. Then  define $\dt(\fraka):=\da(a_1)$ to be the coproduct $\da$ on $A$, giving
\begin{equation}
\dt(1_A)=1_A\ot 1_A.
\mlabel{eq:onea}
\end{equation}
Assume that $\dt(\fraka)$ has been defined for $n$. Consider
 $\fraka=a_1\ot\fraka'\in A^{\ot (n+1)}$ with $\fraka':=a_2\ot \cdots\ot a_n \in A^{\ot n}.$
 By Eq.~(\mref{eq:adr}), we have
\begin{equation}
a_1\ot\fraka'=a_1\dr \pl(\fraka').
\mlabel{eq:drbase}
\end{equation}
By the 1-cocycle property, we first define
 \begin{equation}
 \dt(\pl(\fraka'))=(\id\ot \pl)\dt(\fraka').
 \mlabel{eq:dtpr}
 \end{equation}
Then  define
\begin{equation}
\dt(a_1\ot \fraka')=\dt(a_1)\bul\Big((\id\ot \pl)\dt(\fraka')\Big),
\mlabel{eq:dfndtre}
\end{equation}
 where $\dt(\fraka')$ in Eq.~(\mref{eq:dfndtre}) is well-defined by the induction hypothesis. So $\dt(\fraka)$ is well-defined.

Next, the  counit $\vt$ on $\shat (A)$ will be given in terms of the  counit $\vep_A$ of $A$.
Let $\fraka=a_1\ot a_2\ot \cdots\ot a_n\in A^{\ot n}$ with $n\geq 1$. Define
\begin{equation}
\vt:\shat (A)\to \bfk,\,\fraka\mapsto\vt(\fraka)=\left\{\begin{array}{lll}\vep_A(a_1),& \text{if } \,n=1;\\
0,& \text{if }\, n\geq2.
\end{array}
\right.
\mlabel{eq:dfnvt}
\end{equation}
Then extending by linearity, this map induces a linear map from $\shat(A)$ to $\bfk$. By $\vep_A$ being an algebra homomorphism,  we obtain $\vt(1_A)=\vep_A(1_A)=1_\bfk$.
\begin{lemma}Let $m,n\geq1$ and let $\fraka\in A^{\ot m}$ and $\frakb\in A^{\ot n}$ be pure tensors. Then
\begin{equation}
(\id\ot \dt)(\id\ot \pl)(\fraka\ot \frakb)=(\id\ot \id \ot \pl)(\id \ot \dt)(\fraka\ot \frakb).
\mlabel{eq:comupr1}
\end{equation}
and
\begin{equation}
(\dt\ot \id)(\id\ot \pl)(\fraka\ot \frakb)=(\id\ot\id \ot \pl)(\dt \ot \id)(\fraka\ot \frakb).
\mlabel{eq:comupr2}
\end{equation}
\end{lemma}
\begin{proof} By Eq.~(\mref{eq:dtpr}), we obtain
\begin{eqnarray*}
(\id\ot \dt)(\id\ot \pl)(\fraka\ot \frakb)&=&(\id\ot \dt \pl)(\fraka\ot\frakb)\\
&=&\fraka \ot (\dt \pl(\frakb))\\
&=&\fraka\ot \Big((\id \ot \pl)\dt (\frakb)\Big)\\
&=&(\id\ot \id\ot \pl)(\id \ot \dt)(\fraka\ot \frakb).\\
\end{eqnarray*}
Thus Eq.~(\mref{eq:comupr1}) holds, and Eq.~(\mref{eq:comupr2}) can be done by  straightforward  computation.
\end{proof}
\begin{lemma}
Let  $\shat(A)\ot \shat(A)$, $\bul$ and $\dt$ be as above. Then the triple $(\shat(A)\ot \shat(A), \bul, \id\ot \pl)$ forms a $\lambda$-TD algebra.
\mlabel{lem:doub}
\end{lemma}
\begin{proof}
We only need to show that $\id\ot \pl$ satisfies Eq.~(\mref{eq:LTD}). For all $\fraka\ot \frakb, \frakc\ot \frakd\in \shat(A)\ot \shat(A)$, we have
\begin{eqnarray*}
&&\lefteqn{(\id\ot \pl)(\fraka\ot \frakb)\bul(\id\ot \pl)(\frakc\ot \frakd)}\\
&=&(\fraka\dr\frakc)\ot ((\pl(\frakb)\dr \pl(\frakd))\\
&=&(\fraka\dr \frakc)\ot \pl\Big(\frakb\dr \pl(\frakd)+\pl(\frakb)\dr \frakd+\lambda \frakb\dr\frakd-\frakb\dr \pl(1_A)\dr \frakd\Big)\\
&=&(\id\ot \pl)\Big((\fraka\dr\frakc)\ot \Big(\frakb\dr \pl(\frakd)+\pl(\frakb)\dr \frakd+\lambda \frakb\dr\frakd-(\frakb\dr \pl(1_A))\dr \frakd\Big)\Big)\\
&=&(\id\ot \pl)\Big((\fraka\ot \frakb)\bul(\id\ot \pl)(\frakc\ot\frakd)+(\id\ot \pl)(\fraka\ot \frakb)\bul (\frakc\ot \frakd)\\
&&+\lambda(\fraka\ot\frakb)\bul(\frakc\ot\frakd)-\Big(\fraka\ot (\frakb\dr \pl(1_A))\Big)\bul(\frakc\ot\frakd )\Big)\\
&=&(\id\ot \pl)\Big((\fraka\ot \frakb)\bul(\id\ot \pl)(\frakc\ot\frakd)+(\id\ot \pl)(\fraka\ot \frakb)\bul (\frakc\ot \frakd)\\
&&+\lambda(\fraka\ot\frakb)\bul(\frakc\ot\frakd)-\Big((\fraka\ot\frakb)\bul (1_A\ot \pl(1_A))\Big)\bul(\frakc\ot\frakd )\Big)\\
&=&(\id\ot \pl)\Big((\fraka\ot \frakb)\bul(\id\ot \pl)(\frakc\ot\frakd)+(\id\ot \pl)(\fraka\ot \frakb)\bul (\frakc\ot \frakd)\\
&&+\lambda(\fraka\ot\frakb)\bul(\frakc\ot\frakd)-(\fraka\ot\frakb)\bul (\id\ot \pl)(1_A\ot1_A)\bul(\frakc\ot\frakd )\Big).
\end{eqnarray*}
\end{proof}

\subsection{The compatibilities of $\dt$ and $\vt$}

We are now going to show that $\dt$  and $\vt$ as defined above are compatible with the multiplications.
\mlabel{sec:Compability}

\begin{prop}
The comultiplication $\dt:\shat (A)\to \shat (A)\ot \shat (A)$ is an algebra homomorphism.
\mlabel{prop:comuhom}
\end{prop}

\begin{proof}It suffices to verify that for pure tensors $\fraka\in A^{\ot m}$ and $\frakb\in A^{\ot n}$ with $m,n\geq1$,
\begin{equation}
\dt(\fraka\dr\frakb)=\dt(\fraka)\bul\dt(\frakb).
\mlabel{eq:alghom}
\end{equation}
We prove Eq.~(\mref{eq:alghom}) by induction on $m+n$. If $m+n=2$, then $m=n=1$, and so $\fraka,\frakb\in A$.  By the definitions of $\dr$ and  $\dt$, together with $\Delta_A$ being an algebra homomorphism, Eq.~\eqref{eq:alghom} holds.

Suppose that Eq.~(\mref{eq:alghom}) is true for $m+n\leq k$. Let $m+n=k+1\geq 3$. This leads to either $m\geq 2$ or $n\geq 2$. We just show that Eq.~(\mref{eq:alghom}) holds for the case $m\geq 2$ and $n\geq 2$. The others are similar.  When $m\geq 2$ and $n\geq 2$, denote $\fraka=a_1\ot \fraka'$ with $\fraka'\in A^{\ot (m-1)}$ and $\frakb=b_1\ot \frakb'$ with $\frakb'\in A^{\ot (n-1)}$.  On the one hand,
\begin{eqnarray*}
&&\lefteqn{\dt(\fraka\dr\frakb)}\\
 &=& \dt\bigg(\big(a_1\ot\fraka'\big)\dr\big(b_1\ot\frakb'\big)\bigg)\\
&=&\dt\bigg(\big(a_1\dr \pl(\fraka')\big)\dr \Big(b_1\dr \pl\big(\frakb'\big)\Big)\bigg)\quad(\text{by Eq.~(\mref{eq:drbase})})\\
&=&\dt\bigg(a_1b_1\dr\Big(\pl(\fraka')\dr  \pl(\frakb')\Big)\bigg)\quad(\text{by the definition of $\dr$})\\
&=&\dt\bigg(a_1b_1\dr \pl\Big(\fraka'\dr \pl\big(\frakb'\big)+\pl\big(\fraka'\big)\dr\frakb'+\lambda \fraka'\dr\frakb'-\Big(\fraka'\dr \pl(1_A)\Big) \dr\frakb'\Big)\bigg)\\
&&\quad\quad\quad\quad\quad\quad\quad\quad\quad\quad\quad\quad(\text{by $\pl$ being a $\lambda$-TD operator})\\
&=&\dt\bigg(a_1b_1\ot\Big(\fraka'\dr \pl\big(\frakb'\big)+\pl\big(\fraka'\big)\dr\frakb'+\lambda \fraka'\dr\frakb'-\Big(\fraka'\dr \pl(1_A)\Big) \dr\frakb'\Big)\bigg)
\quad(\text{by Eq.~(\mref{eq:dfndr2})})\\
&=&\dt(a_1b_1)\bul \Bigg((\id\ot \pl)\dt\Big(\fraka'\dr \pl(\frakb')  \Big)
+(\id\ot \pl)\dt\Big(\pl\big(\fraka'\big)\dr\frakb'\Big)\\
&&+\lambda(\id\ot \pl)\dt\big( \fraka'\dr\frakb' \big)
-(\id\ot \pl)\dt\Big((\fraka'\dr \pl(1_A)) \dr\frakb'\Big)\Bigg)\quad(\text{by Eq.~(\mref{eq:dfndtre})})\\
&=&\dt(a_1b_1)\bul\bigg((\id\ot \pl)\Big(\dt\big(\fraka'\big)\bul\dt\big(\pl(\frakb')\big)\Big)
+(\id\ot \pl)\Big(\dt\big(\pl(\fraka')\big)\bul\dt\big(\frakb'\big)\Big)\\
&&+\lambda(\id\ot \pl)\Big(\dt\big( \fraka'\big)\bul\dt\big(\frakb'\big)\Big)
-(\id\ot \pl)\Big(\Big(\dt(\fraka')\bul\dt( \pl(1_A))\Big)\bul\dt(\frakb')\Big)\bigg)\\
&&\quad\quad\quad\quad\quad\quad\quad\quad\quad\quad\quad\quad(\text{by the induction hypothesis})\\
&=&\dt(a_1b_1)\bul\bigg((\id\ot \pl)\Big( \dt(\fraka')\bul (\id \ot \pl)\dt(\frakb')
+(\id\ot \pl)\dt(\fraka')\bul\dt(\frakb')\\
&&+\lambda\dt( \fraka')\bul\dt\big(\frakb'\big)
-\dt(\fraka')\bul(\id\ot \pl)(1_A\ot1_A)\bul\dt(\frakb')\Big)\bigg).
\quad(\text{by  Eqs.~(\mref{eq:onea}) and~(\mref{eq:dtpr})})
\end{eqnarray*}

On the other hand,
\begin{eqnarray*}
&&\lefteqn{\dt(\fraka)\bul\dt(\frakb)}\\
&=&\Big(\dt(a_1)\bul \Big( (\id\ot \pl)\dt(\fraka')\Big)\Big)\bul\Big(\dt(b_1)\bul \Big( (\id\ot \pl)\dt(\frakb')\Big)\Big) \quad(\text{by Eq.~(\mref{eq:dfndtre})})\\
&=&\Big(\dt(a_1)\bul \dt(b_1)\Big)\bul \bigg( (\id\ot \pl)\dt(\fraka')\bul\Big((\id\ot \pl)\dt(\frakb')\Big)\bigg) \quad(\text{by the commutativity of $\bul$})\\
&=&\dt(a_1b_1)\bul \bigg((\id\ot \pl)\Big(\dt(\fraka')\bul(\id\ot \pl)\dt(\frakb')+(\id\ot \pl)\dt(\fraka')\bul \dt(\frakb')\\
&&+\lambda \dt(\fraka')\bul\dt(\frakb')-\dt(\fraka')\bul (\id\ot \pl)(1_A\ot 1_A)\bul \dt(\frakb')\Big)\bigg).\quad(\text{by Lemma~\mref{lem:doub}})
\end{eqnarray*}
Thus the terms of $\dt(\fraka\dr\frakb)$ agree with  the terms of $\dt(\fraka)\bul\dt(\frakb)$, and so Eq.~(\mref{eq:alghom}) holds.  This completes the induction.
\end{proof}
From Proposition~\mref{prop:comuhom}, we obtain
\begin{coro}Let $\shat(A)\ot\shat(A)$ and $\dt$ be as above. Then the induced maps
$$\id\ot\dt:\shat(A)\ot\shat(A)\to \shat(A)\ot\Big(\shat(A)\ot \shat(A)\Big), \fraka\ot\frakb\mapsto \fraka\ot \dt(\frakb)$$
and
$$\dt\ot\id:\shat(A)\ot\shat(A)\to \Big(\shat(A)\ot \shat(A)\Big)\ot\shat(A), \fraka\ot\frakb\mapsto \dt(\fraka)\ot \frakb$$
are algebra homomorphisms, respectively.
\mlabel{coro:indt}
\end{coro}
We next verify that  $\vt:\shat (A)\to \bfk$ given by Eq.~(\mref{eq:dfnvt}) is an algebra homomorphism.
\begin{prop}
The linear map $\vt$ is an algebra homomorphism.
\mlabel{prop:counithom}
\end{prop}
\begin{proof}
By the definition of $\vt$, $\vt(1_A)=1_k$. So we just prove that
 \begin{equation}
 \vt(\fraka \dr\frakb)=\vt(\fraka)\vt(\frakb)
 \mlabel{eq:comvt}
 \end{equation}
 for any pure tensors $\fraka:=a_1\ot \fraka'\in A^{\ot m}$ and
 $\frakb:=b_1\ot\frakb'\in A^{\ot n}$  with $m,n\geq 1$.
 If $m=n=1$, then by Eq.~(\mref{eq:dfndr2}), $\fraka\dr\frakb=a_1b_1$, and so Eq.~(\mref{eq:comvt}) follow from Eq.~(\mref{eq:dfnvt}).
If $m\geq2$ or $n\geq2$, then  $\fraka\dr\frakb=a_1b_1\ot(\fraka'\shar\frakb')$ by Eq.~(\mref{eq:dro}, and so $\fraka\dr\frakb\in \sum_{i\geq 2}^{m+n-1}A^{\ot i}$. Then by Eq.~(\mref{eq:dfnvt}) again,
$$\vt(\fraka\dr\frakb)=0=\vt(\fraka)\vt(\frakb).$$
\end{proof}

\subsection{The coassociativity of $\dt$ and the left  counitality of $\vt$}
\mlabel{sec:Coasso}
In the following, we will show that $\dt$ satisfies the coassociativity and $\vt$ satisfies the left counitality.

\begin{prop}
The comultiplication $\dt$ is coassociative, that is,
\begin{equation}
(\id\ot \dt)\dt=(\dt\ot\id )\dt.
\mlabel{eq:comult}
\end{equation}
\mlabel{prop:comult}
\end{prop}

\begin{proof}
Let $\fraka:=a_1\ot \fraka'\in A^{\ot k}$ with $k\geq 1$.
Then we shall verify that
\begin{equation}
(\id\ot \dt)\dt(\fraka)=(\dt\ot\id )\dt(\fraka).
\mlabel{eq:cop}
\end{equation}
We now proceed by induction on $n$. For $k=1$,  we have $\fraka=a_1\in A$.  Then by the definition of $\dt$ and the coassociativity of $\da$, Eq.~\eqref{eq:cop} holds.

Assume that $k\geq 1$ and Eq.~(\mref{eq:cop}) is true for all  $\fraka\in A^{\ot k}$. Consider $\fraka=a_1\ot \fraka'\in A^{\ot (k+1)}$. Expanding the left hand side $(\id\ot \dt)\dt(\fraka)$ of  Eq.~(\mref{eq:cop}) gives
\begin{eqnarray*}
&&\lefteqn{(\id\ot \dt)\dt(\fraka)}\\
 &=&(\id\ot \dt)\Big(\dt(a_1)\bul\Big((\id\ot \pl)\dt(\fraka')\Big)\Big)\quad(\text{by Eq.~(\mref{eq:dfndtre})})\\
  &=&(\id\ot \dt)\dt(a_1)\bul(\id\ot \dt)(\id\ot \pl)\dt(\fraka')\quad(\text{by Corollary~\mref{coro:indt}})\\
  &=&(\id\ot \dt)\dt(a_1)\bul(\id\ot\id\ot \pl)(\id\ot\dt)\dt(\fraka')\quad(\text{by Eq.~(\mref{eq:comupr1})})\\
 &=&(\id\ot \dt)\dt(a_1)\bul(\id\ot\id\ot \pl)(\dt\ot\id)\dt(\fraka').\quad(\text{by the induction hypothesis})
\end{eqnarray*}
On the other hand, we obtain
\begin{eqnarray*}
&&\lefteqn{(\dt\ot \id)\dt(\fraka)}\\
 &=&(\dt\ot \id)\Big(\dt(a_1)\bul\Big((\id\ot \pl)\dt(\fraka')\Big)\Big)\quad(\text{by Eq.~(\mref{eq:dfndtre})})\\
  &=&(\dt\ot \id)\dt(a_1)\bul(\dt\ot \id)(\id\ot \pl)\dt(\fraka')\quad(\text{by Corollary~\mref{coro:indt}})\\
  &=&(\dt\ot \id)\dt(a_1)\bul(\id\ot\id\ot \pl)(\dt\ot\id)\dt(\fraka')\quad(\text{by Eq.~(\mref{eq:comupr2})})\\
   &=&(\id\ot \dt)\dt(a_1)\bul(\id\ot\id\ot \pl)(\dt\ot\id)\dt(\fraka').\quad(\text{by  the coassociativity of $\da$})
\end{eqnarray*}
Then the expansion of $(\id \ot \dt)\dt(\fraka)$ matches up with the expansion of $(\dt\ot \id)\dt(\fraka)$. This completes the induction, and thus proving Eq.~(\mref{eq:cop}).
\end{proof}

\begin{prop}
The linear map $\vt$  satisfies the  left counitality, that is,
\begin{equation}
(\vt\ot \id)\dt=\beta_\ell,
\end{equation}
where $\beta_\ell:\shat (A)\to \bfk\ot \shat (A)$ is given by $\fraka\mapsto 1\ot \fraka$\, for  $\fraka\in A^{\ot k}$ and for $k\geq 1$.
\mlabel{prop:counit}
\end{prop}
\begin{proof}
It suffices to verify that
\begin{equation}
(\vt\ot \id)\dt(\fraka)=\beta_\ell(\fraka).
\mlabel{eq:vtl}
\end{equation}
for every pure tensor $\fraka\in A^{\ot k}$. We do this by applying the induction on $k\geq 1$. If $k=1$, then $\fraka\in A$, and so Eq.~\mref{eq:vtl} follows from the  left counitality of $\vep_A$.

Assume $k>1$ and consider $\fraka:=a_1\ot \fraka'\in A^{\ot (k+1)}$. Then
\begin{eqnarray*}
(\vt\ot \id )\dt(\fraka)&=&(\vt\ot \id)\dt\Big(a_1\dr \pl( \fraka')\Big)\\
  &=&(\vt\ot \id)\Big(\dt (a_1)\bul \dt(\pl( \fraka'))\Big)\quad(\text{by Eq.~(\mref{eq:dfndtre}}))\\
  &=&\Big((\vt\ot \id)\dt(a_1)\Big)\bigg((\vt\ot \id)\dt\Big(\pl(\fraka')\Big)\bigg)\quad(\text{by Proposition~\mref{prop:counithom}})\\
  &=&\beta_\ell(a_1)(\vt\ot\id)(\id\ot \pl)\dt(\fraka')\quad(\text{by Eq.~(\mref{eq:dtpr})}) \\
  &=&\beta_\ell(a_1)(\id\ot \pl)(\vt\ot\id)\dt(\fraka')\\
  &=&\beta_\ell(a_1)(\id\ot \pl)\beta_\ell(\fraka')\quad(\text{by the induction hypothesis})\\
  &=&\beta_\ell(a_1)\beta_\ell(\pl(\fraka'))\\
  &=&\beta_\ell(a_1\dr \pl(\fraka'))\quad(\text{by $\beta_\ell$ being an algebra isomorphism})\\
&=&\beta_\ell(\fraka).
\end{eqnarray*}
This completes the induction and the proof of Eq.~(\mref{eq:vtl}).
\end{proof}
However, $\vt$ does not satisfy the right counitality. For example, we first define an algebra isomorphism $\beta_r:\shat (A)\to \shat (A)\ot\bfk$, given by $\fraka\mapsto \fraka\ot 1$,\, for all $\fraka\in A^{\ot k}$. Let $\fraka=\pl(x)$, where $x\in A$. By using Sweedler's notation: $\da(x)=\sum x'\ot x''$,  we get
\begin{eqnarray*}
(\id\ot\vt )\dt(\pl(x))&=&(\id\ot\vt )(\id\ot \pl)\da(x)\quad(\text{by Eq.~(\mref{eq:dtpr}}))\\
  &=&(\id\ot\vt )(\id\ot \pl)\Big(\sum x'\ot x''\Big)\\
  &=&\sum x'\ot \vt\Big(\pl( x'')\Big)\\
  &=&0 \quad(\text{by Eq.~(\mref{eq:dfnvt})})\\
  &\neq&\beta_r(\pl(x)).
\end{eqnarray*}
Lastly, we state the main theorem of this section.
It follows that there exists a linear map
$\ml:\bfk\to \shat (A)$, given by
$$c\mapsto c1_A, \quad c\in \bfk.$$
Then we can verify that $\ml$ is a unit for  $(\shat (A),\dr)$.
According to our previous results, we obtain
\begin{theorem}
The sextuple $\shat (A):=(\shat (A),\dr,\ml,\dt,\vt, \pl)$ is a left counital cocycle bialgebra.
\mlabel{thm:lcbia}
\end{theorem}
\begin{proof}
By Theorem~\mref{thm:freeCTD}, the quadruple $(\shat (A),\dr, \ml, \pl)$ is a commutative $\lambda$-TD algebra. Furthermore, the triple $(\shat (A),\dt,\vt)$ is a left counital  coalgebra by  Proposition~\mref{prop:comult} and Proposition~\mref{prop:counit}. Finally,  by Proposition~\mref{prop:comuhom} and Proposition~\mref{prop:counithom}, the sextuple $(\shat (A),\dr,\ml,\dt,\vt,\pl)$ is a left counital cocycle bialgebra.
\end{proof}

\section{The left counital Hopf algebra structure on free commutative $\lambda$-TD algebras}
\mlabel{sec:hopf}

This section will equip the  free commutative $\lambda$-TD algebra $(\shat (A),\dr,\ml,\dt,\vt)$ with a left counital Hopf algebra structure.
\begin{defn} {\rm\cite{Gub,Man}}
\begin{enumerate}
\item A left counital operated bialgebra $H:=(H,m,\mu,\Delta,\vep,P)$ is called {\bf filtered} if there exists an increasing filtration $H^{n}$ for $n\geq 0$ such that
\begin{equation}
 \bigcup\limits_{n\geq0}H^n=H;\quad
H^{p}H^{q}\subseteq H^{p+q};\quad
\Delta(H^{n})\subseteq H^0\ot H^n+\sum_{\substack{p+q=n\\p>0,q>0}} H^p\ot H^q.
\mlabel{eq:defil}
\end{equation}
\item A filtered left counital operated bialgebra $H$ is {\bf connected} if $H^{0}=\im \mu(=\bfk 1_H)$.
\end{enumerate}
\end{defn}
\begin{lemma}Let $\bfk$ be a field. Let $H$ be a connected filtered left counital operated bialgebra and let $e=\mu\vep$.
Then
$$H=\im u\oplus \ker \vep.$$
\mlabel{lem:kerv}
\end{lemma}
\begin{proof}
By  $\vep: H\to \bfk$ being an algebra homomorphism, we obtain $\vep \mu=\id_\bfk$. Then $e^2=\mu(\vep\mu)\vep=e$, and so
$$H=\im e\oplus \ker e.$$
By $e=\mu\vep$, we get $\im e\subseteq \im \mu$.  If $x\in\im \mu$, then $\mu (c)=x$ for some $c\in\bfk$, and so $x=c\mu(1_\bfk)=c\mu(\vep(1_H))=e(c1_H)\in \im e$. Thus $\im e=\im\mu$.  By $e=\mu\vep$ again, $\ker \vep \subseteq \ker e$. Let $z\in \ker e$. Then 
$$e(z)=\mu(\vep(z))=\vep(z)\mu(1_k)=\vep(z)1_H=0.$$
This gives $\vep(z)=0$, and then $\ker e\subseteq \ker \vep$, yielding $\ker e=\ker \vep$. Thus
$$H=\im u\oplus \ker \vep.$$
\end{proof}
By Lemma~\mref{lem:kerv} and the connectedness of $H$, we obtain
\begin{equation}
H=\bfk 1_H\oplus \ker\vep.
\mlabel{eq:conpl}
\end{equation}
\begin{lemma}Let $\bfk$ be a field. Let $H$ be a connected filtered left counital operated bialgebra.
\begin{enumerate}
\item Let $\hat{H}^n:=H^n\cap\ker \vep$ for  $n>0$. Then
\begin{equation}
\hat{H}^n\subseteq \hat{H}^{n+1}
\mlabel{eq:hats}
\end{equation}
and
\begin{equation}
H^n=H^0\oplus  \hat{H}^n.
\mlabel{eq:hath}
\end{equation}
\mlabel{it:hat}
\item Let $p,q>0$. Then
\begin{equation}
H^p\ot H^q\subseteq H^0\ot H^q+\hat{H}^p\ot H^0+ \hat{H}^p\ot\hat{H}^q.
\mlabel{eq:hpq}
\end{equation}
\mlabel{it:hpq}
\item For $x\in \hat{H}^n$ with $n>0$, we have
$$\Delta(x)=1\ot x + \tilde{\Delta}(x),\quad
\text{where}\,\, \tilde{\Delta}(x)\in (\ker\vep\ot H^0+ \ker\vep\ot \ker \vep).$$
\mlabel{it:delx}
\end{enumerate}
\mlabel{lem:delx}
\end{lemma}
\begin{proof}
(\mref{it:hat})
For all $x\in H^n$, we get $x=\vep (x)1_H+x-\vep(x) 1_H$. Since $\vep(x-\vep (x)1_H)=\vep(x)-\vep(x)=0$ and $x-\vep(x)1_H\in H^n+H^0\subseteq H^n$, we have $H^n=H^0+\hat{H}^n$. For every $y\in H^0\cap \hat{H}^n$, we have $y=c 1_H$ for some $c\in\bfk$ by the connectedness of $H$ and $\vep(y)=0$. This leads to $0=\vep(y)=\vep(c 1_H)=c$. Thus $y=0$, proving Eq.~(\mref{eq:hath}).
\smallskip

\noindent
(\mref{it:hpq})Firstly, Eq.~(\mref{eq:hats}) follows from the increasing filtration $H^n\subseteq H^{n+1}$. Secondly, by Eq.~(\mref{eq:hath}), we obtain
\begin{eqnarray*}
H^p\ot H^q&=&(H^0\oplus \hat{H}^p)\ot (H^0\oplus \hat{H}^q)\\
&\subseteq& H^0\ot H^0+H^0\ot \hat{H}^q+\hat{H}^p\ot H^0+\hat{H}^p\ot\hat{H}^q\\
&\subseteq&H^0\ot H^q+\hat{H}^p\ot H^0+\hat{H}^p\ot \hat{H}^q\quad\text{(by $\hat{H}^q\subseteq H^q$ and $H^0\subseteq H^q$})
\end{eqnarray*}
\smallskip

\noindent
(\mref{it:delx})
Let $n>0$. By Eq.~(\mref{eq:defil}), we obtain
\begin{eqnarray*}
\Delta(H^n)&\subseteq& H^0\ot H^n+\sum_{\substack{p+q=n\\p>0,q>0}} H^p\ot H^q\\
&\subseteq& H^0\ot H^n+\sum_{\substack{p+q=n\\p>0,q>0}} H^0\ot H^q+\hat{H}^p\ot H^0+\hat{H}^p\ot \hat{H}^q\quad(\text{by Eq.~(\mref{eq:hpq})})\\
&\subseteq& H^0\ot H^n+\sum_{\substack{p+q=n\\p>0,q>0}} H^0\ot H^q+\sum_{\substack{p+q=n\\p>0,q>0}}\hat{H}^p\ot H^0+\sum_{\substack{p+q=n\\p>0,q>0}}\hat{H}^p\ot \hat{H}^q\\
&\subseteq& H^0\ot H^n+ H^0\ot H^{n-1}+\hat{H}^{n-1}\ot H^0+\sum_{\substack{p+q=n\\p>0,q>0}}\hat{H}^p\ot \hat{H}^q\quad(\text{by Eq.~(\mref{eq:hats})})\\
&\subseteq& H^0\ot H^n+ \hat{H}^{n-1}\ot H^0+\sum_{\substack{p+q=n\\p>0,q>0}}\hat{H}^p\ot \hat{H}^q\\
&\subseteq& H^0\ot H^n+ \ker\vep\ot H^0+\ker\vep\ot\ker\vep.
\end{eqnarray*}
Then for all $x\in \hat{H}^n$ for $n>0$, we can write
$$\Delta(x)=1\ot u+\tilde{\Delta}(x),$$
where $u\in H^n$ and $\tilde{\Delta}(x)\in (\ker\vep\ot H^0+\ker\vep\ot\ker\vep)$. Then by the left counitality  of $\vep$ given by Definition~\mref{defn:leftbia},
\begin{eqnarray*}
x&=&\beta^{-1}(\vep\ot \id)\Delta(x)\\
&=&\beta^{-1}(\vep\ot \id)(1\ot u +\tilde{\Delta}(x))\\
&=&\beta^{-1}(\vep(1)\ot u+(\vep\ot \id)\tilde{\Delta}(x))\\
&=&u.\quad\text{(by $\tilde{\Delta}(x)\in (\ker\vep\ot H^0+\ker\vep\ot\ker\vep)$)}
\end{eqnarray*}
This yields
$$\Delta(x)=1\ot x+\tilde{\Delta}(x),$$
where $\tilde{\Delta}(x)\in (\ker\vep\ot H^0+\ker\vep\ot\ker\vep)$.
\end{proof}
From the above proof of Item~(\mref{it:delx}), we also obtain
\begin{equation}
\tilde{\Delta}(x)\in \sum_{\substack{p+q=n\\p>0,q>0}}H^p\ot H^q.
\mlabel{eq:deltah}
\end{equation}
\begin{defn}Let $H:=(H,m,\mu,\Delta,\vep,P)$ be a  left counital operated bialgebra.
\begin{enumerate}
\item A linear map $S:H\to H$ is said to be an {\bf right antipode} if $S$ is a right inverse of $\id_H$ under the convolution product $\ast$, that is
$$\id_H\ast S=e.$$
\item
A left counital operated bialgebra $H$  with a right antipode is  called a {\bf left counital  Hopf algebra}.
\end{enumerate}
\mlabel{defn:lcHopf}
\end{defn}

The following fact is parallel to  \cite[Corollary II. 3.2]{Man}.
\begin{prop}
  A connected filtered left counital operated bialgebra is a left counital Hopf algebra. The right antipode is recursively defined by
  \begin{equation}
  S(1_H)=1_H,\quad S(x)=-\sum_{x}x'S(x''), \,x\in\ker\vep,
\mlabel{eq:shx}
  \end{equation}
using  Sweedler's notation $\tilde{\Delta}(x)=\sum_{x}x'\ot x''$
  \mlabel{prop:lchopf}
\end{prop}
\begin{proof}
Verify directly that the linear map $S$ defined in Eq.~(\mref{eq:shx}) satisfies the equation $\id \ast S=e$.
By $\Delta$ being an algebra homomorphism, we get $\Delta(1_H)=1_H\ot 1_H$. The formula $e=\mu \vep$ gives $$e(1_H)=\mu(\vep(1_H))=\mu(1_\bfk)=1_H.$$
Then
$$(\id \ast S)(1_H)=m(\id\ot S)\Delta(1_H)=S(1_H)\Rightarrow (\id \ast S)(1_H)=1_H=e(1_H).$$
Let $x\in \ker\vep$. Then by Lemma~\mref{lem:delx} Item~(\mref{it:delx})),
$$(\id \ast S)(x)=m(\id\ot S)\Delta(x)=m(\id\ot S)(1\ot x+\tilde{\Delta}(x))=S(x)+\sum_{x}x'S(x''),$$
where $\tilde{\Delta}(x)=\sum_{x}x'\ot x''\in  \sum_{\substack{p+q=n\\p>0,q>0}}H^p\ot H^q$ follows immediately from Eq.~(\mref{eq:deltah}).
By Eq.~(\mref{eq:shx}), we obtain
 $$S(x)+\sum_{x} x'S(x'')=0.$$
This gives
$$(\id \ast S)(x)=0=e(x), \,x\in\ker\vep.$$
\end{proof}
\begin{theorem}
Let $A=\bigcup\limits_{n\geq0}A^n$ is a connected filtered left counital bialgebra. Let $\shat(A)=(\shat (A),\dr,\ml,\dt,\vt, \pl)$ be as in Theorem~\mref{thm:lcbia}.  Then $\shat(A)$ is a left counital Hopf algebra.
\mlabel{thm:Main}
\end{theorem}
\begin{proof}
According to Proposition~\mref{prop:lchopf}, we only need to verify that $\shat(A)$ is a connected filtered left counital operated bialgebra.
For this reason, we denote the degree of $a$ by
$$\deg(a):=\min\{k\in\NN \,|\,a\in A^k\},\quad \forall a\in A.$$
For any $m\geq 1$ and any pure tensor $0\neq\fraka=a_1\ot\cdots\ot a_m\in A^{\ot m}$, we set
\begin{equation}\deg(\fraka) := \deg(a_1)+\cdots+\deg(a_m)+m-1.
\mlabel{eq:dega}
\end{equation}
For simplicity, we write $\lam :=\shat(A)=\oplus_{n\geq 1} A^{\ot n}$ and denote by $\lam^k$  the linear span of pure tensors $\fraka\in\lam$ with $\deg(\fraka)\leq k$. Then we get a increasing filtration $\lam^k\subseteq\lam^{k+1}$ for all $k\geq 0$ and $\lam^0=A^0=\bfk 1_A$ by the connectedness of $A$. Let $\fraka=a_1\ot\fraka'\in A^{\ot(n+1)}$ with $\fraka'\in A^{\ot n}$. Then  by Eq.~(\mref{eq:dega})  we obtain
 \begin{equation}
\deg(\fraka)=\deg(a_1)+\deg(\fraka')+1.
\mlabel{eq:fra1}
\end{equation}
Furthermore, if $\fraka\in\lam^r$ for $r\geq1$, then
\begin{equation}
\fraka'\in\lam^{r-{\rm deg} (a_1)-1}.
\mlabel{eq:degfra}
\end{equation}

We next show that the increasing filtration $\lam^k$  satisfies that for all $p,q\geq 0$
\begin{equation} \lam^q\dr\lam^p\subseteq \lam^{p+q}
\mlabel{eq:lpq}
\end{equation}
and
\begin{equation}
\dt(\lam^k)\subseteq \lam^0\ot\lam^k+\sum_{\substack{p+q=k\\p>0,q>0}}\lam^p\ot\lam^q.
\mlabel{eq:lamsub}
\end{equation}
Now use induction on $p+q\geq 0$ to verify Eq.~(\mref{eq:lpq}). For this it suffices to prove  $\fraka\dr\frakb\in\lam^{p+q}$ for all pure tensors $\fraka\in\lam^p$ and $\frakb\in\lam^q$. When $p+q=0$, we have $\fraka,\frakb\in\lam^0$, and so $\fraka\dr\frakb\in\lam^0$ by $\lam^0=\bfk 1_A$ and Eq.~(\mref{eq:dfndr2}).  Assume that Eq.~(\mref{eq:lpq}) holds for $p+q\leq n$. Let $p+q=n+1$. If $p=0$ or $q=0$, then $\fraka\in\lam^0$ or $\frakb\in\lam^0$, proving Eq.~(\mref{eq:lpq}) by Eq.~(\mref{eq:dfndr2}) again. Hence we suppose that $p, q\geq1$. If $\fraka\in A$ or $\frakb\in A$, then $\deg(\fraka\dr\frakb)\leq \deg(\fraka)+\deg(\frakb)$ by Eq.~(\mref{eq:dfndr2}) and the connectedness of $A$, and so Eq.~(\mref{eq:lpq}) holds. Thus we only consider $\fraka\in A^{\ot \ell}$ and $\frakb\in A^{\ot m}$  for $\ell, m\geq2$.
Write $\fraka=a_1\ot\fraka'$ with $\fraka'=a_2\ot\cdots\ot a_\ell$, and $\frakb=b_1\ot\frakb'$ with $\frakb'=b_2\ot\cdots\ot b_m$.  By Eq.~(\mref{eq:dfndr2}) again, we obtain
\begin{equation}
\fraka\dr\frakb =a_1b_1\ot (\fraka' \dr (1_A\ot \frakb')) +a_1b_1\ot ((1_A\ot \fraka')\dr \frakb')
+\lambda a_1b_1\ot( \fraka'\dr\frakb') - a_1b_1\ot \Big(\big(\fraka' \dr \pl(1_A)\big)\dr \frakb'\Big).
\mlabel{eq:frakab}
\end{equation}
By Eq.~(\mref{eq:degfra}), $\fraka'\in\lam^{p-\deg(a_1)-1}$ and $\frakb'\in\lam^{q-\deg(b_1)-1}$. Furthermore, by Eq.~(\mref{eq:fra1}) and $\deg(1_A)=0$ because $\lam^0=\bfk 1_A$, we have
$$\deg(1_A\ot \fraka')=\deg(1_A)+\deg(\fraka')+1=\deg(\fraka')+1\Rightarrow 1_A\ot \fraka'\in \lam^{p-\deg(a_1)}$$
and
$$\deg(1_A\ot \frakb')=\deg(1_A)+\deg(\frakb')+1=\deg(\frakb')+1\Rightarrow 1_A\ot \frakb'\in \lam^{q-\deg(b_1)}.$$
Since $p-\deg(a_1)-1+q-\deg(b_1)=p+q-\deg(a_1)-\deg(b_1)-1<p+q$, we have $ \fraka' \dr (1_A\ot \frakb')\in \lam^{p+q-\deg(a_1)-\deg(b_1)-1}$ by the induction hypothesis. Thus
\begin{eqnarray*}
\deg(a_1b_1\ot(\fraka'\dr(1_A\ot \frakb')))&=&\deg(a_1b_1)+\deg(\fraka'\dr (1_A\ot \frakb'))+1\\
&\leq & \deg(a_1)+\deg(b_1)+p+q-\deg(a_1)-\deg(b_1)-1+1\\
&=&p+q.
\end{eqnarray*}
This gives $a_1b_1\ot(\fraka'\dr(1_A\ot \frakb'))\in \lam^{p+q}$. Similarly, $a_1b_1\ot((1_A\ot\fraka')\dr \frakb')\in \lam^{p+q}$ and $\lambda a_1b_1\ot( \fraka'\dr\frakb')\in \lam^{p+q-1}$. For the fourth term on the right-hand side of Eq.~(\mref{eq:frakab}), by $\deg(P(1_A))=\deg(1_A)+\deg(1_A)+1=1$  and the induction hypothesis, we obtain
$$ \fraka'\dr P(1_A)\in \lam^{p-\deg(a_1)},$$
and thus using the induction hypothesis yields $\big((\fraka' \dr \pl(1_A))\dr \frakb'\big)\in \lam^{p+q-\deg(a_1)-\deg(b_1)-1}$, thereby proving
$$a_1b_1\ot \Big(\big(\fraka' \dr \pl(1_A)\big)\dr \frakb'\Big)\in\lam^{p+q}.$$
Hence all terms on the right-hand side of Eq.~(\mref{eq:frakab}) are in $\lam^{p+q}$,  yielding $\fraka\dr\frakb\in\lam^{p+q}$.

Finally, it remains to prove  Eq.~(\mref{eq:lamsub}).  The proof proceeds by induction on $k\geq 0$, with the case $k=0$ is  true, because $\lam^0=\bfk 1_A$ and $\dt(1_A)=1_A\ot 1_A$. Assume that $k\geq 0$ and Eq.~(\mref{eq:lamsub}) holds for all pure tensors $\fraka\in\lam^k$. Consider  $\fraka\in \lam^{k+1}$. If $\fraka\in A(=\cup_{n\geq 0}A^n)$,  then by $\deg(\fraka)\leq k+1$, we get $\fraka\in A^{k+1}$. Since $A$ is a connected filtered left counital bialgebra and $A^n\subseteq \lam^n$ for all $n\geq 0$, we have
\begin{eqnarray*}
\dt(\fraka)=\da(\fraka)\in A^0\ot A^{k+1}+\sum_{\substack{p+q=k+1\\p>0,q>0}} A^p\ot A^q\subseteq \lam^0\ot\lam^{k+1}+\sum_{\substack{p+q=k+1\\p>0,q>0}} \lam^p\ot \lam^q.
\end{eqnarray*}
We then suppose that $\fraka=a_1\ot \fraka'\in A^{\ot \ell+1}$ with $\fraka'\in A^{\ell}$ for $\ell\geq 1$. Then
\begin{eqnarray*}
\dt(\fraka)&=&\dt(a_1\ot \fraka')\\
&=&\dt(a_1)\bul\big((\id\ot \pl)\dt(\fraka')\big)\quad(\text{by Eq.~(\mref{eq:dfndtre})})\\
&=&\da(a_1)\bul\big((\id\ot \pl)\dt(\fraka')\big).
\end{eqnarray*}
By $\fraka\in\lam^{k+1}$ and  Eq.~(\mref{eq:degfra}), we get $\fraka'\in \lam^{k+1-\deg(a_1)-1}=\lam^{k-\deg(a_1)}$. Then applying the induction hypothesis gives
$$\dt(\fraka')\in \lam^0\ot \lam^{k-\deg(a_1)}+\sum_{\substack{p_2+q_2=k-\deg(a_1)\\p_2>0,q_2>0}}\lam^{p_2}\ot\lam^{q_2}.$$
Thus
\begin{eqnarray*}
\dt(\fraka)&=&\da(a_1)\bul\big((\id\ot \pl)\dt(\fraka')\big)\\
&\in&\Big(\lam^0\ot \lam^{\deg(a_1)}+\sum_{\substack{p_1+q_1=\deg(a_1)\\p_1>0,q_1>0}}\lam^{p_1}\ot\lam^{q_1}\Big)\\
&&\bul\Big(\id\ot\pl\Big)\Big(\lam^0\ot\lam^{k-\deg(a_1)}+\sum_{\substack{p_2+q_2=k-\deg(a_1)\\p_2>0,q_2>0}}\lam^{p_2}\ot\lam^{q_2}\Big)\\
&\subseteq&\Big(\lam^0\ot \lam^{\deg(a_1)}+\sum_{\substack{p_1+q_1=\deg(a_1)\\p_1>0,q_1>0}}\lam^{p_1}\ot\lam^{q_1}\Big)\\
&&\bul\Big(\lam^0\ot\lam^{k-\deg(a_1)+1}+\sum_{\substack{p_2+q_2=k-\deg(a_1)\\p_2>0,q_2>0}}\lam^{p_2}\ot\lam^{q_2+1}\Big)\\
&\subseteq&\lam^0\ot\lam^{k+1}+\sum_{\substack{p_2+q_2=k-\deg(a_1)\\p_2>0,q_2>0}}\lam^{p_2}\ot\lam^{\deg(a_1)+q_2+1}\\
&&+\sum_{\substack{p_1+q_1=\deg(a_1)\\p_1>0,q_1>0}}\lam^{p_1}\ot\lam^{q_1+k-\deg(a_1)+1}
+\sum_{\substack{p_2+q_2=k-\deg(a_1)\\p_1+q_1=\deg(a_1)\\p_1>0,q_1>0,p_2>0,q_2>0}}\lam^{p_1+p_2}\ot\lam^{q_1+q_2+1}\\
&\subseteq&\lam^0\ot\lam^{k+1}
+\sum_{\substack{p_2+q_2=k-\deg(a_1)\\p_1+q_1=\deg(a_1)\\p_1\geq0,q_1>0,p_2\geq0,q_2>0}}\lam^{p_1+p_2}\ot\lam^{q_1+q_2+1}\quad(\text{ $p_1^2+p_2^2\neq 0$})\\
&\subseteq&\lam^0\ot\lam^{k+1}+\sum_{\substack{p+q=k+1\\p>0,q>0}}\lam^p\ot\lam^q\quad (\text{$p:=p_1+p_2,\,q:=q_1+q_2+1$}).
\end{eqnarray*}
This completes the induction and thus proves Eq.~(\mref{eq:lamsub}).
\end{proof}

\smallskip

\noindent {\bf Acknowledgements}: This work was supported by the National Natural Science Foundation of
China (Grant No.~11601199 and 11961031).

\end{document}